\documentclass[12pt]{article}
\pdfoutput=1
\usepackage[utf8]{inputenc}
\usepackage{amsmath}
\usepackage{amssymb}
\usepackage{tikz}
\usepackage{physics}
\usepackage{graphicx}
\usepackage{array}
\usepackage{enumerate}
\usepackage{appendix}
\usepackage[margin=1in]{geometry}
\usepackage{ragged2e}
\usepackage[shortlabels]{enumitem}
\usepackage{dsfont}
\usepackage{mathtools}
\usepackage{biblatex}
\usepackage{appendix}
\usepackage{hyperref}
\usepackage{xurl}
\hypersetup{breaklinks=true}
\usepackage{authblk}
\usepackage{multicol}
\usepackage{enumerate}
\usepackage{soul}
\usepackage{algorithm}
\usepackage{algpseudocode}
\def\last{\texttt{last}}
\def\seq{\texttt{seq}}
\DeclareMathOperator{\Em}{Em}
\DeclareMathOperator{\newEm}{newEm}

\usetikzlibrary{decorations.pathreplacing}

\usepackage{hyperref}
\usepackage{xurl}
\def\S{\mathfrak{S}}
\def\inv{^{-1}}

\def\ph{\varphi}

\def\ci{\subseteq}
\def\cs{\supseteq}
\def\r{\epsilon}
\def\recip{\,\stackrel{R}{\text{---}}\,}
\def\S{\mathfrak{S}}
\def\inv{^{-1}}

\def\ph{\varphi}

\def\ph{\phi}
\def\r{\epsilon}

\everymath{\displaystyle}
\newcommand{\al}[1]{\begin{align*} #1 \end{align*}}

\usepackage{lipsum}
\usepackage{amsthm}
\usepackage{amsmath}

\newtheorem{prop}{Proposition}

\theoremstyle{definition}
\newtheorem{example}{Example}

\newtheorem{theorem}{Theorem}
\newtheorem{lemma}{Lemma}

\newtheorem{definition}{Definition}
\newtheorem{cor}{Corollary}
\newcommand{\remove}[1]{}
\newcommand{\delete}[1]{}

\title{Bijections for generalized Wilf equivalences}
\author{Melanie Ferreri}

\begin{document}
\maketitle
\begin{abstract}
Starting with an inclusion-exclusion proof of a combinatorial identity, a direct bijection can be produced using recursive subtraction (sometimes with a direct combinatorial description). We apply this method to identities for generalized Wilf equivalences among consecutive patterns in inversion sequences, giving direct bijective proofs of some generalized Wilf equivalences shown by Auli and Elizalde. We also give new bijective proofs of a stronger relation among some consecutive patterns.
\end{abstract}

\section{Introduction}

The 
study of patterns in inversion sequences was initiated by Corteel, Martinez, Savage, and Weselcouch \cite{Corteel_2016}, as well as Mansour and Shattuck \cite{mansourshattuck}. 
Auli and Elizalde \cite{aulielizalde, aulielizaldetwo} \nobreak{analogously} studied \textit{consecutive} patterns in inversion sequences, and characterized generalized Wilf equivalences among consecutive patterns of length 4. In many cases they gave direct bijective proofs. In others, they gave a collection of bijections among supersets of the desired sets, and showed the equivalences by applying inclusion-exclusion.

A proof of this type can be ``bijectivized'' by recursively subtracting the maps between supersets, until we have mappings directly between the exact desired sets. For those equivalences which they do not prove with direct bijections, we extend their work by using recursively defined bijections to produce matchings which exhibit the equivalences.

We describe the recursive maps obtained via their maps, and give an iterative interpretation for
certain cases, which yields a bijective proof of these equivalences. 

For the remaining equivalences not covered in the previous case, we show the recursive map in the same way; although for this case, we did not obtain a simpler direct description of the resulting map. However, the resulting examples suggest that a stronger result can be shown. We find a new bijective algorithm which proves a stronger relation among patterns that implies super strong Wilf equivalence. 

The takeaway is that in this setting, it is often worth considering the implicit bijection that results from the inclusion-exclusion proof. As we’ve seen in this case, this can lead to a direct description of a bijection showing the result, and can give insight toward possible stronger results.

Much of the content of this paper is from a project described in the author’s PhD thesis, \cite{ferreri}.

\section*{Acknowledgement}
The author would like to thank Peter Doyle for numerous helpful discussions, insights, and suggestions. Many thanks as well to Sergi Elizalde, Peter Winkler, and Yan Zhuang for their valuable feedback regarding this work.

\section[Equivalence relations]{Equivalence relations for patterns in inversion sequences}

Let $\pi = \pi(1)\pi(2)\cdots\pi(n) \in \S_n$. An \textit{inversion} is a pair $(i,j)$ such that $i<j$ and $\pi(i) > \pi(j)$; that is, a pair of indices where the elements of $[n]$ are ``out of order.'' An \textit{inversion sequence} is an integer sequence
$$
\r = \r_1 \r_2 \cdots \r_n
$$
such that $0 \leq \r_i < i$ for all $1 \leq i \leq n$. We denote by $I_n$ the set of inversion sequences of length $n$. This set bijects with $\S_n$; for a permutation $\pi$, we can view the $i$th entry of an inversion sequence as counting the number of entries to the left of $\pi(i)$ which are greater than $\pi(i)$.

A \textit{pattern} is a sequence 
$$
p = p_1 p_2 \cdots p_r
$$
with $p_i \in \{0, 1, \ldots, r-1\}$ for all $1 \leq i \leq r$. We also require that $j$ can only appear in $p$ if $j-1$ also appears.

If $w = w_1 w_2 \cdots w_k$ is a word over the integers, its \textit{reduction} is given by replacing all occurrences of its $i$th smallest entry by $i-1$ for all $i$. 

A sequence $\r$ \textit{contains} the pattern $p$ if there is a subsequence $\r_{i_1}, \r_{i_2}, \ldots, \r_{i_r}$ whose reduction is $p$. Otherwise, $\r$ \textit{avoids} $p$.
An inversion sequence $\r$ contains the \textit{consecutive pattern} $p = p_1p_2\ldots p_r$ if there is a consecutive subsequence of $\r$ whose reduction is $p$. We denote that a pattern is consecutive by underlining it. If the first term of the subsequence is $\r_i$, we say the subsequence is an \textit{occurrence} of $p$ in position $i$, or simply that $i$ is an occurrence of $p$ in $\r$.
We denote by $I_n(p)$ the set of inversion sequences of length $n$ which avoid $p$.

\begin{example}
   $\r= 00123021$ contains $\underline{012}$. It also contains $210$ (witness: $\r_5\r_7\r_8$), but avoids $\underline{210}$. That is, it has a decreasing subsequence of length 3 but not a consecutive decreasing subsequence of length 3. So $\r \in I_8(\underline{210})$.
\end{example}

Denote by $\Em(p,\r)$ the set of all occurrences of $p$ in $\r$, so $\Em(p,\r)$ is a list of indices that tell us the locations of the first positions of all consecutive subsequences with reduction $p$ appearing in $\r$. 
The consecutive patterns $p$, $q$ are said to be
\begin{itemize}
\item \textit{Wilf equivalent} if $| I_n(p)| = | I_n(q)|$ for all $n$.
\item \textit{strongly Wilf equivalent} if for all $n$ and $m$, the number of inversion sequences of length $n$ containing $m$ occurrences of $p$ is equal to the number of inversion sequences of length $n$ containing $m$ occurrences of $q$. 

\item \textit{super strongly Wilf equivalent} if
$$
|\{ \r\in I_n : \Em(p,\r) = T \}| = |\{ \r\in I_n : \Em(q,\r) = T \}|
$$
for all $n$, $T \subseteq [n]$. 

\end{itemize}

In \cite{aulielizalde}, Auli and Elizalde show generalized Wilf equivalences between all 75 consecutive patterns of length 4. 
It turns out that all of their demonstrated equivalences among consecutive patterns of length 4 are super strong, and they conjecture that for consecutive patterns of length $k$, Wilf equivalence implies strong Wilf equivalence.
Here we will write $p \sim q$ when $p$ and $q$ are super strongly Wilf equivalent.

A consecutive pattern $p$ is \textit{non-overlapping} if it cannot overlap itself in more than one entry. That is, if $p$ is a pattern of length $n$, no two occurrences of $p$ can be less than $n-1$ apart.

\begin{example}
\begin{enumerate}
    \item The consecutive pattern $\underline{0102}$ is not non-overlapping; the sequence $010203$ has an occurrence of this pattern at positions $1$ and $3$; these occurrences overlap in two entries (the third and fourth entries in the sequence).
    \item The consecutive pattern $\underline{1000}$ is non-overlapping. Suppose we want to construct a sequence $\r$ with $\underline{1000}$ occurring in both positions $i$ and $i+1$. Since the pattern occurs at $i$, we have $\r_i>\r_{i+1} = \r_{i+2} = \r_{i+3}$. But also, since the pattern occurs at $i+1$, we have $\r_{i+1} > \r_{i+2}$; which is a contradiction. Similarly, we could not have this pattern occurring at both $i$ and $i+2$. So the pattern can overlap itself in at most one entry.
\end{enumerate}
    
\end{example}

We say the consecutive patterns $p$ and $q$ of length $n$ are \textit{mutually non-overlapping} if occurrences of $p$ and $q$ in an inversion sequence cannot overlap in more than one entry. 

We also consider exchanging occurrences of $p$ for occurrences of $q$ in a sequence. The process of changing an occurrence of $p$ to an occurrence of $q$ is given in \cite{aulielizalde} as follows: Let $p = \underline{p_1 p_2\cdots p_r}$ and $q = \underline{q_1 q_2 \cdots q_r}$ be consecutive patterns which agree on their first and last entries, and which both contain all elements in $\{0,1,\ldots, d\}$ for some $d$, which is the maximal entry in both sequences. Let $\r \in I_n$ such that $\r$ has an occurrence of $p$ at position $i$. Then define $$f: \{0,\ldots, d\} \to \{\r_i, \ldots, \r_{i+r-1}\}$$ (with the latter set excluding repetition) to be the order-preserving bijection such that $$\r_i \cdots \r_{i+r-1} = f(p_1) \cdots f(p_r).$$ Then define a new sequence $\r' = \r_1' \cdots \r_n'$ where $\r_j' = \r_j$ for $j \neq i, \ldots i+r-1$; i.e.\ when $j$ is not a position that coincides with the occurrence of $p$ at position $i$. For the remaining positions, assign $$\r_i' \r_{i+1}' \cdots \r_{i+r-1}' = f(q_1) f(q_2) \cdots f(q_r).$$ This operation is a \textit{change} of the occurrence of $p$ at position $i$ in $\r$ to an occurrence of $q$. If the resulting sequence $\r'$ is an inversion sequence, this change is said to be \textit{valid}.

\begin{example}
    Let $p = \underline{0102}$, $q = \underline{0112}$, and $\r= 0101213$. The sequence $\r$ has an occurrence of $p$ at position $4$. To change this to an occurrence of $q$, we let $f(0) = 1$, $f(1) = 2$, and $f(2) = 3$. Then we write $\r' = 010 f(0) f(1) f(1) f(2) = 0101223$. This yields a sequence with an occurrence of $q$ at position $4$.
\end{example}

A pattern $p = \underline{p_1\ldots p_r}$ is said to be \textit{changeable} for $q= \underline{q_1\ldots q_r}$ if, for all $i \in \{1,\ldots, r \}$, 
$$
q_i \leq \max ( \{p_j : 1 \leq j \leq i\} \cup \{ p_j-j+1 : i< j \leq r\}).
$$
Two patterns $p, q$ are \textit{interchangeable} if they are changeable for one another. 

Lemma 4.8 of \cite{aulielizalde} shows that $p$ is changeable for $q$ if and only if, for any inversion sequence $\r$ and any occurrence of $p$ in $\r$, it is valid to change the occurrence of $p$ to an occurrence of $q$. In other words, one can always change an occurrence of $p$ 
to an occurrence of $q$ while keeping $\r$ as an inversion sequence. 

For patterns which are non-overlapping, mutually non-overlapping, and interchangeable, Auli and Elizalde \cite{aulielizalde} show their Wilf equivalence via a bijection which simultaneously changes all occurrences of $p$ into $q$ and vice versa --- this can be done at once because the occurrences cannot possibly overlap. This takes care of the equivalences (ii) -- (vii), (ix), (xi), (xii), and (xiv) in Theorem 2.2 of \cite{aulielizalde} (also included in Theorem \ref{asupdated} below). 

This actually proves an even stronger statement than super strong Wilf equivalence. {These bijections show
$$
|\{\r \in I_n : \Em(p,\r) = S, \Em(q,\r) = T\}| = |\{\r\in I_n : \Em(p,\r) = T, \Em(q,\r) = S\}|
$$
for all $n$ and $S,T \ci [n]$. When two patterns $p,q$ satisfy this relation, we will say that $p$ and $q$ are \textit{reciprocal},
and we will denote this by $p \recip
 q$. 
As pointed out in \cite{aulielizalde}, this is equivalent to saying that the sets of positions where $p$ and $q$ occur have a symmetric joint distribution over $I_n$.
Although $\recip$ 
is not an equivalence relation, we have that if $p$ and $q$ are reciprocal, then 
$p \sim q$.
We will revisit this definition in Section \ref{sec:superduper}.}

After we are done, we will have the following updated list of equivalences and reciprocal relations, where those with asterisks are newly proven bijectively:

\begin{samepage}
\begin{theorem}[Extension of Theorem 2.2 of \cite{aulielizalde}]
\label{asupdated}
A complete list of the generalized Wilf equivalences between consecutive patterns of length 4 is as follows:

\begin{multicols}{2}
\begin{enumerate}[(i)]
\item * $\underline{0102} \recip \underline{0112}.$

\item $\underline{0021}\sim \underline{0121}.$

\item $\underline{1002}\sim \underline{1012}\sim \underline{1102}.$

\item $\underline{0100}\sim \underline{0110}.$

\item $\underline{2013}\sim \underline{2103}.$

\item $\underline{1200}\sim \underline{1210}\sim \underline{1220}.$

\item $\underline{0211}\sim\underline{0221}.$
\columnbreak

\item * $\underline{1000}\sim \underline{1110}$.

\item $\underline{1001}\sim\underline{1011}\sim\underline{1101}$.

\item * $\underline{2100}\sim \underline{2210}$.

\item $\underline{2001}\sim \underline{2011}\sim \underline{2101}\sim \underline{2201}$.

\item $\underline{2012}\sim \underline{2102}$.

\item * $\underline{2010}\recip \underline{2110}\recip\underline{2120}$.

\item $\underline{3012}\sim \underline{3102}$.
\end{enumerate}
\end{multicols}

All of the other (super strong) Wilf equivalence classes contain just one pattern.
\end{theorem}

\end{samepage}

For the equivalences without asterisks demonstrated by Auli and Elizalde, the patterns are shown to be reciprocal as well, as explained in \cite{aulielizalde}.
The remaining (super strong) Wilf equivalences are proven by Auli and Elizalde in \cite{aulielizalde} using an inclusion-exclusion argument with  bijections $\{ \r \in I_n : \Em(p,\r) \supseteq S\}\to\{ \r \in I_n : \Em(q,\r) \supseteq S\}$ which change occurrences of $p$ at positions $S$ to occurrences of $q$. 

Applying the method of bijective inclusion-exclusion as described in \cite{doyle:lrs}, we can use these maps to obtain the direct bijection recursively.

\section{Bijections defined recursively}

We will first focus on nonoverlapping consecutive patterns; in particular, the equivalences $\underline{1000} \sim \underline{1110}$ and $\underline{2100} \sim \underline{2210}$.
For consecutive patterns $p$ and $q$, we are interested in finding a map which takes a sequence $\r$, with $p$ occurring exactly at a set of indices $T$, and generate a new sequence $\r'$, which has $q$ occurring at exactly the indices in $T$.

For a consecutive pattern $p$ and set $T \ci [n]$, let
$$I_{n,p}(\supseteq T) = \{\r \in I_n : \Em(p,\r) \supseteq T\}.$$ 
Similarly define 
$$I_{n,p}(= T) = \{\r \in I_n : \Em(p,\r) = T\}.$$

Suppose $p$ and $q$ are nonoverlapping consecutive patterns, and suppose there is a bijection $$\ph_{\geq T}: I_{n,p}(\supseteq T) \to I_{n,q}(\supseteq T)$$ which changes the occurrences of $p$ at positions $T$ into occurrences of $q$ at positions $T$.
For our purposes, the bijection $\ph_{\geq T}$ will be given by applying the change operation of Auli and Elizalde \cite{aulielizalde} described in Section 1 to each occurrence of $p$ indexed by $T$.

For each $\r \in I_{n,p}(= T)$, we can define $\ph_{= T}(\r)$ recursively as follows:

The base cases occur when $T$ is a maximal set of occurrences for $p$ or $q$; that is, when there is no inversion sequence whose set of occurrences of $p$ or $q$ properly contains $T$.
{In this case, the map $\ph_{\geq T}$ will actually be a matching between $I_{n,p}(= T)$ and $I_{n,q}(= T)$, since $T$ is maximal so $I_{n,p}(\supseteq T) = I_{n,p}(= T)$.}

For a non-maximal set $T$, suppose we have bijections $\ph_{=S}$ for all $S$ properly containing $T$. Then define 
\begin{align*}
\ph_{= T}: (\r \in I_{n,p}(=T)) \xmapsto{\ph_{\geq T}
} 
(\r' \in I_{n,q}(=S) \ci I_{n,q}(\supseteq T) ) \xmapsto{\ph^{-1}_{= S}} (\r'' \in I_{n,p}(= S) ) &\\
\xmapsto{\ph_{\geq T}} (\r''' \in I_{n,q}(= R) )
\xmapsto{\,\ph^{-1}_{= R}, \ph_{\geq T}}_{\text{(repeat until)}} ( \r {''} {}^{\cdots} {'} 
\in I_{n,q}(= T)) &
\end{align*}
where $S = \Em(q, \ph_{\geq T}(\r))$; 
i.e. the set of occurrences of $q$ in the image of $\r$ after applying $\ph_{\geq T}$, and $R$ is defined similarly as the set of occurrences of $q$ in the sequence after applying the most recent arrow. By the hypothesis, each map applied is a bijection, so this gives a matching between $I_{n,p}(=T) $ and $I_{n,q}(=T)$.

Note: This process computes the subtraction as in \cite{doyle:lrs} of the maps $\phi_{= S}$ for all $S \supsetneq T$ from $\phi_{\geq T}$. A more in-depth explanation of this, along with some examples, can be found in Section 1.1 of \cite{ferreri}.

\begin{example} Suppose $p = \underline{1000}$ and $q = \underline{1110}$. The following examples illustrate how the recursive mapping works. The vertical lines indicate recursion depth, with the leftmost line corresponding to a depth of 0.
    \begin{enumerate}[(a)]
    \item $\r = 001000$ \\ 
We currently have $p$ at $\{ 3\}$; we want $q$ at $\{ 3\}$. \\ 
$\ph_{ = \{ 3\} } (001000)$: \\ 
\hphantom{...}$\mid$\,\,$\ph_{ \geq \{ 3\} } (001000) = 001110$ \\ 
\hphantom{...}$\mid$\,\,$q$ occurs at $\{ 3\}$, so we are done. \\ 
$001110$ \\ 
We finally obtain $001000 \mapsto 001110$.  \\

    Since there was only one occurrence and it worked to just swap it, this works the same way on this sequence as the other bijections would.

    \item 
    $\r = 0021110$ \\ 
We currently have $p$ at $\{ 3\}$; we want $q$ at $\{ 3\}$. \\ 
$\ph_{ = \{ 3\} } (0021110)$: \\ 
\hphantom{...}$\mid$\,\,$\ph_{ \geq \{ 3\} } (0021110) = 0022210$ \\ 
\hphantom{...}$\mid$\,\,$q$ occurs at $\{ 3\}$, so we are done. \\ 
$0022210$ \\ 
We finally obtain $0021110 \mapsto 0022210$.  \\

    The initial sequence had another copy of $q$ at position 4, but this got overwritten. Going the opposite direction, $0022210$ maps to $0021110$ (we just change the $q$ to a $p$).

    \item $\r = 002111110$ \\ 
We currently have $p$ at $\{ 3\}$; we want $q$ at $\{ 3\}$. \\ 
$\ph_{ = \{ 3\} } (002111110)$: \\ 
\hphantom{...}$\mid$\,\,$\ph_{ \geq \{ 3\} } (002111110) = 002221110$ \\ 
\hphantom{...}$\mid$\,\,$q$ occurs at $\{ 3, 6\}$, so apply $\ph\inv_{ = \{ 3, 6\} }$ \\ 
\hphantom{...}$\mid$\,\,$\ph\inv_{ = \{ 3, 6\} } (002221110)$: \\ 
\hphantom{...}$\mid$\,\,\hphantom{...}$\mid$\,\,$\ph\inv_{ \geq \{ 3, 6\} } (002221110) = 002111000$ \\ 
\hphantom{...}$\mid$\,\,\hphantom{...}$\mid$\,\,$p$ occurs at $\{ 3, 6\}$, so we are done. \\ 
\hphantom{...}$\mid$\,\,$002111000$ \\ 
\hphantom{...}$\mid$\,\,$\ph_{ \geq \{ 3\} } (002111000) = 002221000$ \\ 
\hphantom{...}$\mid$\,\,$q$ occurs at $\{ 3\}$, so we are done. \\ 
$002221000$ \\ 
We finally obtain $002111110 \mapsto 002221000$.  \\ 

The following pictures show the underdiagonal lattice paths of the original sequence, and the sequences after each application of $\ph_{\geq \{3\}}$.

\begin{figure}[H]
    \centering
 \begin{tikzpicture}[scale = .5] 
\draw[thick, green] (5.1, -.3 ) to[out=-30,in= 210] (8.9, -.3); 
\draw[green] (7,-1.3) node{$q$}; 
\filldraw[green, opacity=.3] (5, 0) -- (5, 1) -- (6, 1) -- (6,0); 
\filldraw[green, opacity=.3] (6, 0) -- (6, 1) -- (7, 1) -- (7,0); 
\filldraw[green, opacity=.3] (7, 0) -- (7, 1) -- (8, 1) -- (8,0); 
\filldraw[green, opacity=.3] (8, 0) -- (8, 0) -- (9, 0) -- (9,0); 
\filldraw[blue, opacity=.3] (2, 0) -- (2, 2) -- (3, 2) -- (3,0); 
\filldraw[blue, opacity=.3] (3, 0) -- (3, 1) -- (4, 1) -- (4,0); 
\filldraw[blue, opacity=.3] (4, 0) -- (4, 1) -- (5, 1) -- (5,0); 
\filldraw[blue, opacity=.3] (5, 0) -- (5, 1) -- (6, 1) -- (6,0); 
\draw[thick, blue] (2.1, -.3 ) to[out=-30,in= 210] (5.9, -.3); 
\draw[blue] (4,-1.3) node{$p$}; 
\draw (0,0) -- (1, 0); 
\draw (1,0) -- (1, 0) -- (2, 0); 
\draw (2,0) -- (2, 2) -- (3, 2); 
\draw (3,2) -- (3, 1) -- (4, 1); 
\draw (4,1) -- (4, 1) -- (5, 1); 
\draw (5,1) -- (5, 1) -- (6, 1); 
\draw (6,1) -- (6, 1) -- (7, 1); 
\draw (7,1) -- (7, 1) -- (8, 1); 
\draw (8,1) -- (8, 0) -- (9, 0); 
\end{tikzpicture}, \hfill
\begin{tikzpicture}[scale = .5] 
\filldraw[blue, opacity=.3] (4, 0) -- (4, 2) -- (5, 2) -- (5,0); 
\filldraw[blue, opacity=.3] (5, 0) -- (5, 1) -- (6, 1) -- (6,0); 
\filldraw[blue, opacity=.3] (6, 0) -- (6, 1) -- (7, 1) -- (7,0); 
\filldraw[blue, opacity=.3] (7, 0) -- (7, 1) -- (8, 1) -- (8,0); 
\draw[thick, blue] (4.1, -.3 ) to[out=-30,in= 210] (7.9, -.3); 
\draw[blue] (6,-1.3) node{$p$}; 
\draw[thick, green] (2.1, -.3 ) to[out=-30,in= 210] (5.9, -.3); 
\draw[green] (4,-1.3) node{$q$}; 
\filldraw[green, opacity=.3] (2, 0) -- (2, 2) -- (3, 2) -- (3,0); 
\filldraw[green, opacity=.3] (3, 0) -- (3, 2) -- (4, 2) -- (4,0); 
\filldraw[green, opacity=.3] (4, 0) -- (4, 2) -- (5, 2) -- (5,0); 
\filldraw[green, opacity=.3] (5, 0) -- (5, 1) -- (6, 1) -- (6,0); 
\draw[thick, green] (5.1, -.3 ) to[out=-30,in= 210] (8.9, -.3); 
\draw[green] (7,-1.3) node{$q$}; 
\filldraw[green, opacity=.3] (5, 0) -- (5, 1) -- (6, 1) -- (6,0); 
\filldraw[green, opacity=.3] (6, 0) -- (6, 1) -- (7, 1) -- (7,0); 
\filldraw[green, opacity=.3] (7, 0) -- (7, 1) -- (8, 1) -- (8,0); 
\filldraw[green, opacity=.3] (8, 0) -- (8, 0) -- (9, 0) -- (9,0); 
\draw[thick, dashed, blue] (0,0) -- (1, 0); 
\draw[thick, dashed, blue] (1,0) -- (1, 0) -- (2, 0); 
\draw[thick, dashed, blue] (2,0) -- (2, 2) -- (3, 2); 
\draw[thick, dashed, blue] (3,2) -- (3, 2) -- (4, 2); 
\draw[thick, dashed, blue] (4,2) -- (4, 2) -- (5, 2); 
\draw[thick, dashed, blue] (5,2) -- (5, 1) -- (6, 1); 
\draw[thick, dashed, blue] (6,1) -- (6, 1) -- (7, 1); 
\draw[thick, dashed, blue] (7,1) -- (7, 1) -- (8, 1); 
\draw[thick, dashed, blue] (8,1) -- (8, 0) -- (9, 0); 
\end{tikzpicture}, \hfill
\begin{tikzpicture}[scale = .5] 
\filldraw[blue, opacity=.3] (5, 0) -- (5, 1) -- (6, 1) -- (6,0); 
\filldraw[blue, opacity=.3] (6, 0) -- (6, 0) -- (7, 0) -- (7,0); 
\filldraw[blue, opacity=.3] (7, 0) -- (7, 0) -- (8, 0) -- (8,0); 
\filldraw[blue, opacity=.3] (8, 0) -- (8, 0) -- (9, 0) -- (9,0); 
\draw[thick, blue] (5.1, -.3 ) to[out=-30,in= 210] (8.9, -.3); 
\draw[blue] (7,-1.3) node{$p$}; 
\draw[thick, green] (2.1, -.3 ) to[out=-30,in= 210] (5.9, -.3); 
\draw[green] (4,-1.3) node{$q$}; 
\filldraw[green, opacity=.3] (2, 0) -- (2, 2) -- (3, 2) -- (3,0); 
\filldraw[green, opacity=.3] (3, 0) -- (3, 2) -- (4, 2) -- (4,0); 
\filldraw[green, opacity=.3] (4, 0) -- (4, 2) -- (5, 2) -- (5,0); 
\filldraw[green, opacity=.3] (5, 0) -- (5, 1) -- (6, 1) -- (6,0); 
\draw[thick, blue] (0,0) -- (1, 0); 
\draw[thick,  blue] (1,0) -- (1, 0) -- (2, 0); 
\draw[thick, blue] (2,0) -- (2, 2) -- (3, 2); 
\draw[thick,  blue] (3,2) -- (3, 2) -- (4, 2); 
\draw[thick, blue] (4,2) -- (4, 2) -- (5, 2); 
\draw[thick,  blue] (5,2) -- (5, 1) -- (6, 1); 
\draw[thick, blue] (6,1) -- (6, 0) -- (7, 0); 
\draw[thick, blue] (7,0) -- (7, 0) -- (8, 0); 
\draw[thick, blue] (8,0) -- (8, 0) -- (9, 0); 
\end{tikzpicture} 

    \caption{Underdiagonal lattice path diagrams of sequences $002111110$, $002221110$, $002221000$.}
\end{figure}
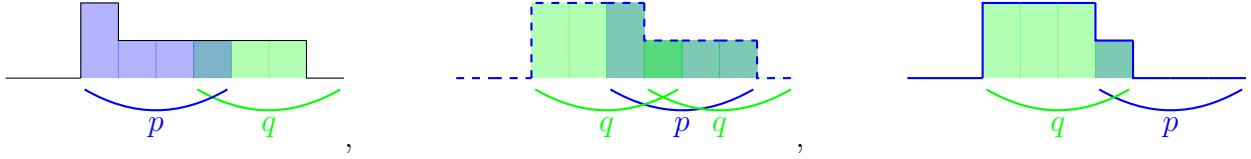

    \item \begin{samepage}

    $\r = 0100011111110$ \\ 
We currently have $p$ at $\{ 2\}$; we want $q$ at $\{ 2\}$. \\ 
$\ph_{ = \{ 2\} } (0100011111110)$: \\ 
\hphantom{...}$\mid$\,\,$\ph_{ \geq \{ 2\} } (0100011111110) = 0111011111110$ \\ 
\hphantom{...}$\mid$\,\,$q$ occurs at $\{ 2, 10\}$, so apply $\ph\inv_{ = \{ 2, 10\} }$ \\ 
\hphantom{...}$\mid$\,\,$\ph\inv_{ = \{ 2, 10\} } (0111011111110)$: \\ 
\hphantom{...}$\mid$\,\,\hphantom{...}$\mid$\,\,$\ph\inv_{ \geq \{ 2, 10\} } (0111011111110) = 0100011111000$ \\ 
\hphantom{...}$\mid$\,\,\hphantom{...}$\mid$\,\,$p$ occurs at $\{ 2, 10\}$, so we are done. \\ 
\hphantom{...}$\mid$\,\,$0100011111000$ \\ 
\hphantom{...}$\mid$\,\,$\ph_{ \geq \{ 2\} } (0100011111000) = 0111011111000$ \\ 
\hphantom{...}$\mid$\,\,$q$ occurs at $\{ 2, 8\}$, so apply $\ph\inv_{ = \{ 2, 8\} }$ \\ 
\hphantom{...}$\mid$\,\,$\ph\inv_{ = \{ 2, 8\} } (0111011111000)$: \\ 
\hphantom{...}$\mid$\,\,\hphantom{...}$\mid$\,\,$\ph\inv_{ \geq \{ 2, 8\} } (0111011111000) = 0100011100000$ \\ 
\hphantom{...}$\mid$\,\,\hphantom{...}$\mid$\,\,$p$ occurs at $\{ 2, 8\}$, so we are done. \\ 
\hphantom{...}$\mid$\,\,$0100011100000$ \\ 
\hphantom{...}$\mid$\,\,$\ph_{ \geq \{ 2\} } (0100011100000) = 0111011100000$ \\ 
\hphantom{...}$\mid$\,\,$q$ occurs at $\{ 2, 6\}$, so apply $\ph\inv_{ = \{ 2, 6\} }$ \\ 
\hphantom{...}$\mid$\,\,$\ph\inv_{ = \{ 2, 6\} } (0111011100000)$: \\ 
\hphantom{...}$\mid$\,\,\hphantom{...}$\mid$\,\,$\ph\inv_{ \geq \{ 2, 6\} } (0111011100000) = 0100010000000$ \\ 
\hphantom{...}$\mid$\,\,\hphantom{...}$\mid$\,\,$p$ occurs at $\{ 2, 6\}$, so we are done. \\ 
\hphantom{...}$\mid$\,\,$0100010000000$ \\ 
\hphantom{...}$\mid$\,\,$\ph_{ \geq \{ 2\} } (0100010000000) = 0111010000000$ \\ 
\hphantom{...}$\mid$\,\,$q$ occurs at $\{ 2\}$, so we are done. \\ 
$0111010000000$ \\ 
We finally obtain $0100011111110 \mapsto 0111010000000$.  \\ 
    \end{samepage}
    
The following pictures show the underdiagonal lattice paths of the original sequence, and the sequences after each application of $\ph_{\geq \{2\}}$.

    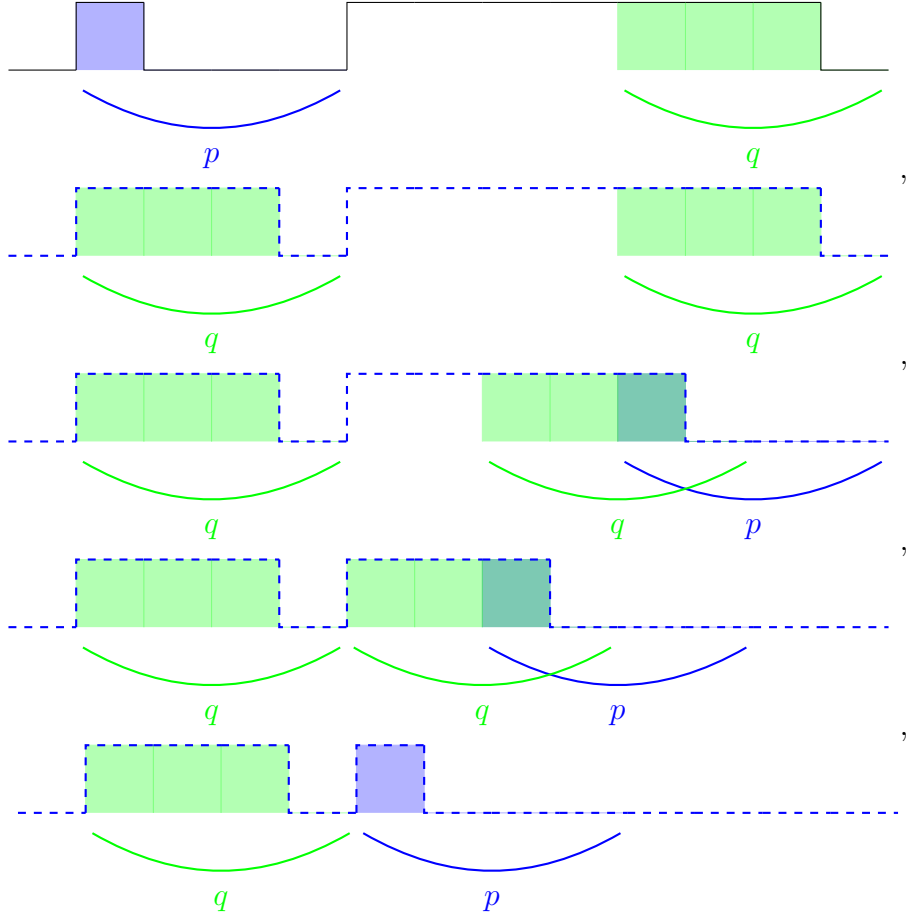
\begin{figure}[H]
        \centering
        \begin{tikzpicture}[scale = 0.9] 
\draw[thick, green] (9.1, -.3 ) to[out=-30,in= 210] (12.9, -.3); 
\draw[green] (11,-1.3) node{$q$}; 
\filldraw[green, opacity=.3] (9, 0) -- (9, 1) -- (10, 1) -- (10,0); 
\filldraw[green, opacity=.3] (10, 0) -- (10, 1) -- (11, 1) -- (11,0); 
\filldraw[green, opacity=.3] (11, 0) -- (11, 1) -- (12, 1) -- (12,0); 
\filldraw[green, opacity=.3] (12, 0) -- (12, 0) -- (13, 0) -- (13,0); 
\filldraw[blue, opacity=.3] (1, 0) -- (1, 1) -- (2, 1) -- (2,0); 
\filldraw[blue, opacity=.3] (2, 0) -- (2, 0) -- (3, 0) -- (3,0); 
\filldraw[blue, opacity=.3] (3, 0) -- (3, 0) -- (4, 0) -- (4,0); 
\filldraw[blue, opacity=.3] (4, 0) -- (4, 0) -- (5, 0) -- (5,0); 
\draw[thick, blue] (1.1, -.3 ) to[out=-30,in= 210] (4.9, -.3); 
\draw[blue] (3,-1.3) node{$p$}; 
\draw (0,0) -- (1, 0); 
\draw (1,0) -- (1, 1) -- (2, 1); 
\draw (2,1) -- (2, 0) -- (3, 0); 
\draw (3,0) -- (3, 0) -- (4, 0); 
\draw (4,0) -- (4, 0) -- (5, 0); 
\draw (5,0) -- (5, 1) -- (6, 1); 
\draw (6,1) -- (6, 1) -- (7, 1); 
\draw (7,1) -- (7, 1) -- (8, 1); 
\draw (8,1) -- (8, 1) -- (9, 1); 
\draw (9,1) -- (9, 1) -- (10, 1); 
\draw (10,1) -- (10, 1) -- (11, 1); 
\draw (11,1) -- (11, 1) -- (12, 1); 
\draw (12,1) -- (12, 0) -- (13, 0); 
\end{tikzpicture} ,
 
\begin{tikzpicture}[scale = .9] 
\draw[thick, green] (1.1, -.3 ) to[out=-30,in= 210] (4.9, -.3); 
\draw[green] (3,-1.3) node{$q$}; 
\filldraw[green, opacity=.3] (1, 0) -- (1, 1) -- (2, 1) -- (2,0); 
\filldraw[green, opacity=.3] (2, 0) -- (2, 1) -- (3, 1) -- (3,0); 
\filldraw[green, opacity=.3] (3, 0) -- (3, 1) -- (4, 1) -- (4,0); 
\filldraw[green, opacity=.3] (4, 0) -- (4, 0) -- (5, 0) -- (5,0); 
\draw[thick, green] (9.1, -.3 ) to[out=-30,in= 210] (12.9, -.3); 
\draw[green] (11,-1.3) node{$q$}; 
\filldraw[green, opacity=.3] (9, 0) -- (9, 1) -- (10, 1) -- (10,0); 
\filldraw[green, opacity=.3] (10, 0) -- (10, 1) -- (11, 1) -- (11,0); 
\filldraw[green, opacity=.3] (11, 0) -- (11, 1) -- (12, 1) -- (12,0); 
\filldraw[green, opacity=.3] (12, 0) -- (12, 0) -- (13, 0) -- (13,0); 
\draw[thick, dashed, blue] (0,0) -- (1, 0); 
\draw[thick, dashed, blue] (1,0) -- (1, 1) -- (2, 1); 
\draw[thick, dashed, blue] (2,1) -- (2, 1) -- (3, 1); 
\draw[thick, dashed, blue] (3,1) -- (3, 1) -- (4, 1); 
\draw[thick, dashed, blue] (4,1) -- (4, 0) -- (5, 0); 
\draw[thick, dashed, blue] (5,0) -- (5, 1) -- (6, 1); 
\draw[thick, dashed, blue] (6,1) -- (6, 1) -- (7, 1); 
\draw[thick, dashed, blue] (7,1) -- (7, 1) -- (8, 1); 
\draw[thick, dashed, blue] (8,1) -- (8, 1) -- (9, 1); 
\draw[thick, dashed, blue] (9,1) -- (9, 1) -- (10, 1); 
\draw[thick, dashed, blue] (10,1) -- (10, 1) -- (11, 1); 
\draw[thick, dashed, blue] (11,1) -- (11, 1) -- (12, 1); 
\draw[thick, dashed, blue] (12,1) -- (12, 0) -- (13, 0); 
\end{tikzpicture} ,
 
\begin{tikzpicture}[scale = .9] 
\filldraw[blue, opacity=.3] (9, 0) -- (9, 1) -- (10, 1) -- (10,0); 
\filldraw[blue, opacity=.3] (10, 0) -- (10, 0) -- (11, 0) -- (11,0); 
\filldraw[blue, opacity=.3] (11, 0) -- (11, 0) -- (12, 0) -- (12,0); 
\filldraw[blue, opacity=.3] (12, 0) -- (12, 0) -- (13, 0) -- (13,0); 
\draw[thick, blue] (9.1, -.3 ) to[out=-30,in= 210] (12.9, -.3); 
\draw[blue] (11,-1.3) node{$p$}; 
\draw[thick, green] (1.1, -.3 ) to[out=-30,in= 210] (4.9, -.3); 
\draw[green] (3,-1.3) node{$q$}; 
\filldraw[green, opacity=.3] (1, 0) -- (1, 1) -- (2, 1) -- (2,0); 
\filldraw[green, opacity=.3] (2, 0) -- (2, 1) -- (3, 1) -- (3,0); 
\filldraw[green, opacity=.3] (3, 0) -- (3, 1) -- (4, 1) -- (4,0); 
\filldraw[green, opacity=.3] (4, 0) -- (4, 0) -- (5, 0) -- (5,0); 
\draw[thick, green] (7.1, -.3 ) to[out=-30,in= 210] (10.9, -.3); 
\draw[green] (9,-1.3) node{$q$}; 
\filldraw[green, opacity=.3] (7, 0) -- (7, 1) -- (8, 1) -- (8,0); 
\filldraw[green, opacity=.3] (8, 0) -- (8, 1) -- (9, 1) -- (9,0); 
\filldraw[green, opacity=.3] (9, 0) -- (9, 1) -- (10, 1) -- (10,0); 
\filldraw[green, opacity=.3] (10, 0) -- (10, 0) -- (11, 0) -- (11,0); 
\draw[thick, dashed, blue] (0,0) -- (1, 0); 
\draw[thick, dashed, blue] (1,0) -- (1, 1) -- (2, 1); 
\draw[thick, dashed, blue] (2,1) -- (2, 1) -- (3, 1); 
\draw[thick, dashed, blue] (3,1) -- (3, 1) -- (4, 1); 
\draw[thick, dashed, blue] (4,1) -- (4, 0) -- (5, 0); 
\draw[thick, dashed, blue] (5,0) -- (5, 1) -- (6, 1); 
\draw[thick, dashed, blue] (6,1) -- (6, 1) -- (7, 1); 
\draw[thick, dashed, blue] (7,1) -- (7, 1) -- (8, 1); 
\draw[thick, dashed, blue] (8,1) -- (8, 1) -- (9, 1); 
\draw[thick, dashed, blue] (9,1) -- (9, 1) -- (10, 1); 
\draw[thick, dashed, blue] (10,1) -- (10, 0) -- (11, 0); 
\draw[thick, dashed, blue] (11,0) -- (11, 0) -- (12, 0); 
\draw[thick, dashed, blue] (12,0) -- (12, 0) -- (13, 0); 
\end{tikzpicture} ,
 
\begin{tikzpicture}[scale = .9] 
\filldraw[blue, opacity=.3] (7, 0) -- (7, 1) -- (8, 1) -- (8,0); 
\filldraw[blue, opacity=.3] (8, 0) -- (8, 0) -- (9, 0) -- (9,0); 
\filldraw[blue, opacity=.3] (9, 0) -- (9, 0) -- (10, 0) -- (10,0); 
\filldraw[blue, opacity=.3] (10, 0) -- (10, 0) -- (11, 0) -- (11,0); 
\draw[thick, blue] (7.1, -.3 ) to[out=-30,in= 210] (10.9, -.3); 
\draw[blue] (9,-1.3) node{$p$}; 
\draw[thick, green] (1.1, -.3 ) to[out=-30,in= 210] (4.9, -.3); 
\draw[green] (3,-1.3) node{$q$}; 
\filldraw[green, opacity=.3] (1, 0) -- (1, 1) -- (2, 1) -- (2,0); 
\filldraw[green, opacity=.3] (2, 0) -- (2, 1) -- (3, 1) -- (3,0); 
\filldraw[green, opacity=.3] (3, 0) -- (3, 1) -- (4, 1) -- (4,0); 
\filldraw[green, opacity=.3] (4, 0) -- (4, 0) -- (5, 0) -- (5,0); 
\draw[thick, green] (5.1, -.3 ) to[out=-30,in= 210] (8.9, -.3); 
\draw[green] (7,-1.3) node{$q$}; 
\filldraw[green, opacity=.3] (5, 0) -- (5, 1) -- (6, 1) -- (6,0); 
\filldraw[green, opacity=.3] (6, 0) -- (6, 1) -- (7, 1) -- (7,0); 
\filldraw[green, opacity=.3] (7, 0) -- (7, 1) -- (8, 1) -- (8,0); 
\filldraw[green, opacity=.3] (8, 0) -- (8, 0) -- (9, 0) -- (9,0); 
\draw[thick, dashed, blue] (0,0) -- (1, 0); 
\draw[thick, dashed, blue] (1,0) -- (1, 1) -- (2, 1); 
\draw[thick, dashed, blue] (2,1) -- (2, 1) -- (3, 1); 
\draw[thick, dashed, blue] (3,1) -- (3, 1) -- (4, 1); 
\draw[thick, dashed, blue] (4,1) -- (4, 0) -- (5, 0); 
\draw[thick, dashed, blue] (5,0) -- (5, 1) -- (6, 1); 
\draw[thick, dashed, blue] (6,1) -- (6, 1) -- (7, 1); 
\draw[thick, dashed, blue] (7,1) -- (7, 1) -- (8, 1); 
\draw[thick, dashed, blue] (8,1) -- (8, 0) -- (9, 0); 
\draw[thick, dashed, blue] (9,0) -- (9, 0) -- (10, 0); 
\draw[thick, dashed, blue] (10,0) -- (10, 0) -- (11, 0); 
\draw[thick, dashed, blue] (11,0) -- (11, 0) -- (12, 0); 
\draw[thick, dashed, blue] (12,0) -- (12, 0) -- (13, 0); 
\end{tikzpicture} ,
 
\begin{tikzpicture}[scale = .9] 
\filldraw[blue, opacity=.3] (5, 0) -- (5, 1) -- (6, 1) -- (6,0); 
\filldraw[blue, opacity=.3] (6, 0) -- (6, 0) -- (7, 0) -- (7,0); 
\filldraw[blue, opacity=.3] (7, 0) -- (7, 0) -- (8, 0) -- (8,0); 
\filldraw[blue, opacity=.3] (8, 0) -- (8, 0) -- (9, 0) -- (9,0); 
\draw[thick, blue] (5.1, -.3 ) to[out=-30,in= 210] (8.9, -.3); 
\draw[blue] (7,-1.3) node{$p$}; 
\draw[thick, green] (1.1, -.3 ) to[out=-30,in= 210] (4.9, -.3); 
\draw[green] (3,-1.3) node{$q$}; 
\filldraw[green, opacity=.3] (1, 0) -- (1, 1) -- (2, 1) -- (2,0); 
\filldraw[green, opacity=.3] (2, 0) -- (2, 1) -- (3, 1) -- (3,0); 
\filldraw[green, opacity=.3] (3, 0) -- (3, 1) -- (4, 1) -- (4,0); 
\filldraw[green, opacity=.3] (4, 0) -- (4, 0) -- (5, 0) -- (5,0); 
\draw[thick, dashed, blue] (0,0) -- (1, 0); 
\draw[thick, dashed, blue] (1,0) -- (1, 1) -- (2, 1); 
\draw[thick, dashed, blue] (2,1) -- (2, 1) -- (3, 1); 
\draw[thick, dashed, blue] (3,1) -- (3, 1) -- (4, 1); 
\draw[thick, dashed, blue] (4,1) -- (4, 0) -- (5, 0); 
\draw[thick, dashed, blue] (5,0) -- (5, 1) -- (6, 1); 
\draw[thick, dashed, blue] (6,1) -- (6, 0) -- (7, 0); 
\draw[thick, dashed, blue] (7,0) -- (7, 0) -- (8, 0); 
\draw[thick, dashed, blue] (8,0) -- (8, 0) -- (9, 0); 
\draw[thick, dashed, blue] (9,0) -- (9, 0) -- (10, 0); 
\draw[thick, dashed, blue] (10,0) -- (10, 0) -- (11, 0); 
\draw[thick, dashed, blue] (11,0) -- (11, 0) -- (12, 0); 
\draw[thick, dashed, blue] (12,0) -- (12, 0) -- (13, 0); 
\end{tikzpicture} 
 
        \caption[Underdiagonal lattice path diagrams.]{Underdiagonal lattice path diagrams of sequences $0100011111110$, $0111011111110$, $0111011111000$, $0111011100000$, $0111010000000$.}
    \end{figure}
\end{enumerate}
\end{example}

\subsection{Iterative interpretation}

For nonoverlapping consecutive patterns, we can translate this into an iterative process. 

\begin{lemma}\label{lem:qtop}
    Let $\ph\inv_{\geq S}$ be as previously defined, and let $\r \in I_n$ such that $\Em(q,\r) = S$ and $\Em(p,\r) = U$. If $p$ and $q$ are non-overlapping consecutive patterns of length 4, then $\Em(p,\ph\inv_{\geq S}(\r)) \ci S \cup U$.
\end{lemma}
\begin{proof}
    We will show that applying $\ph\inv_{\geq S}$ (i.e., changing $q$ at positions $S$ to $p$) to $\r$ cannot create new occurrences of $p$ outside of $S$ or $U$. 
    
    Applying $\ph\inv_{\geq S}$ only affects the sequence entries coinciding with the pattern occurrences indexed by $S$, by changing those occurrences of $q$ to occurrences of $p$. Suppose for contradiction that there was an occurrence $p$ not indexed by $U$ that was created in a position also not indexed by $S$. Then it would necessarily overlap the same entries as the $q$ indexed by $S$ that were previous there and were changed to $p$. (In the case of the consecutive patterns of length 4, each pair of non-overlapping Wilf equivalent patterns agree on the first and fourth entries, so the only entries changed would be the second or third entries of the previous $q$. Thus the overlap would have occurred on the second or third entries). This implies that we would have two copies of $p$ overlapping in more than one entry, which contradicts the assumption that $p$ is non-overlapping. 
    So there cannot have been any ``extra'' occurrences of $p$ created outside of $S$ (or $U$, where there may have been preexisting occurrences of $p$); thus $\Em(p,\ph\inv_{\geq S}(\r)) \ci S \cup U$.
\end{proof}

\begin{lemma}\label{lem:recdepth}
    Let $\ph_{=T}$ be the previously defined recursive bijection, and let $\r \in I_n$ with $\Em(p,\r) = T$. If $p$ and $q$ are non-overlapping consecutive patterns of length 4, then the recursion depth will not exceed 1.
\end{lemma}

\begin{proof}
    We will show that 
    after (i) turning all $p$ at positions $T$ to $q$, and then (ii) turning all $q$ at positions $S \cs T$ to $p$, there will be no $p$ at a position outside of $S$ remaining in the sequence.

Suppose to the contrary that after step (ii), there is still some $p$ at position $r \notin S$. Since $T \ci S$, this means $r \notin T$, so this occurrence of $p$ at $r$ was not in the original sequence. Then it was either created on step (i) when $p$ were changed to $q$ or on step (ii) when $q$ were changed to $p$.

In step (i), all $p$ were changed to $q$. If this had created an extra occurrence of $p$ in the sequence, then the new occurrences of $p$ would necessarily overlap one of the newly created occurrences of $q$ at a position indexed by $T$. Then since the $q$ overlapping $p$ at $r$ is at a position in $T \ci S$, in step (ii) this $q$ would then be changed to a $p$, erasing the $p$ at position $r$ (e.g. $\underline{1000}00 \rightarrow 11\underline{1000}0 \rightarrow \underline{1000}00$). We know that this $p$ will be erased because if it was preserved, then we would have two overlapping instances of $p$, but $p$ is non-overlapping.

Otherwise it must be the case that the new occurrence of $p$ at $r$ was created on step (ii). By Lemma \ref{lem:qtop}, it is not possible for a new $p$ to appear after step (ii) that was not either already in the sequence after step (i) or in a position where a $q$ was after step (i). We have just checked that any newly created $p$ from step (i) will be erased after step (ii),  and the 
positions where $q$ appeared after step (i) are all indexed by $S$, so step (ii) cannot have created a new $p$ at a position outside of $S$.

Thus, if a sequence $\r$ had $q$ at $S \cs T$ after the first arrow, then $\r$ has $p$ at exactly $S$ after the second arrow.
So $\ph_{=S}\inv(\r)$ will have been computed without increasing the recursion depth. 
\end{proof}

\begin{theorem} \label{thm:iter}
    Let $p$ and $q$ be non-overlapping consecutive patterns of length 4. The previously defined recursive bijection $\ph$ matches the following iterative description.
    Let $\r \in I_n$.
    \begin{enumerate}
    \item Find $T:= \Em(p,\r)$.
    \item Change all occurrences of $p$ in $T$ to $q$; let $\r'$ be the resulting sequence.
    \item Now let $S = \Em(q,\r')$.
    Change all of the occurrences of $q$ in $S \setminus T$ to $p$.
    \item Repeat step (c) until $S \setminus T = \emptyset$, so we now have a sequence $\r'$ with $q$ occurring exactly at $T$.
\end{enumerate}
\end{theorem}

\begin{proof}
By Lemma \ref{lem:recdepth}, the recursion depth will never exceed 1 when computing $\ph(\r)$. So the mapping will continue to apply 
$\ph_{=S}\inv$ and $\ph_{\geq T}$,
for each new set 
$S$ containing $T$,
until $T = S$.

Moreover, the initial occurrences of $p$ at $T$ will swap back and forth between $p$ and $q$ at each step. We know that the original occurrences of $p$ will not be erased because this could only happen if some $q$ appeared at position $i$ overlapping the initial $p$ and then was changed to a $p$, erasing some original occurrence. However, if this $q$ appeared in the original sequence, then the first arrow would change the $p$ to $q$, erasing the $q$ at position $i$. 

{The $q$ could not have appeared on the second arrow by the same reasoning as in the proof of Lemma \ref{lem:recdepth} (changing $p$ to $q$ cannot create new occurrences of $q$ outside of where the $p$ were).}

Therefore, the original occurrences of $p$ at $T$ will switch back and forth, and at the end will turn into occurrences of $q$ 
at $T$, so we can view the bijection as changing the initial occurrences of $p$ to occurrences of $q$ and then changing only entries off of $T$. 
This matches exactly the iterative description.
\end{proof}

The previous theorem gives a direct bijective proof of the super-strong Wilf equivalences $\underline{1000} \sim \underline{1110}$ and $\underline{2100} \sim \underline{2210}$.

\begin{example} Again let $p = \underline{1000}$, $q = \underline{1110}$.
The sequence $\r = 002111110$ has $p$ at $T = \{3\}$. We change this to $002221110$ and have $q$ at $S = \{3,6\}$. Then $S \setminus T = \{ 6\}$, so we change the $q$ at $6$ back to $p$, and obtain $002221000$. This now has $S = \{3\} = T$, so we are done.\vspace{3mm} 

    Going the other direction, $002221000$ has $q$ at $\{3\}$, so we change it to $p$ and obtain $002111000$. Now we have $p$ at $\{3,6\}$ so we change the $p$ at $6$ back to $q$ and obtain $002111110$.

\end{example}

\section[Not mutually non-overlapping]{Patterns that are not mutually non-overlapping} \label{sec:01020112case}

    The remaining equivalences are
    (i) $\underline{0102}\sim\underline{0112}$ and 
    (xiii) $ \underline{2010}\sim\underline{2110}\sim\underline{2120}.$
       Auli and Elizalde prove these equivalences similarly as before, by constructing maps which exchange the patterns, and then arguing by inclusion-exclusion. The maps which exchange the patterns are slightly more complicated since these patterns can overlap themselves. The bijection for swapping overlapping copies of the pattern $\underline{0102}$ to overlapping copies of a pattern $\underline{0112}$ is given by Auli and Elizalde \cite{aulielizalde}, as follows: 
 Since $\underline{0102}$ can overlap itself in two entries, a string of overlapping occurrences of this pattern might look something like: $01020304\ldots$.
 To get the pattern $\underline{0112}$ instead, this would then be changed into $01122334\ldots$.
  Viewing these sequences as underdiagonal lattice paths, this can be represented by making the second ``0'' of each occurrence match the height of the entry to its left.
This process changes all occurrences of $\underline{0102}$ in a sequence to $\underline{0112}$. 
For the relation (xiii), they state that similarly constructed maps can be used for changing between $\underline{2010}$, $\underline{2110}$ and $\underline{2120}$.

We can still define a bijection recursively for this case, the same way as before. However, the iterative algorithm from Theorem \ref{thm:iter} does not match the recursive process in this case. For example, if we try to replicate the same iterative process we used before on $\r = 010223$, we would obtain $011223$, which has $q$ at $\{1,3\}$, and then changing $q$ at $3$ back to $p$ gives $011213$, so now we have $q$ only at $\{1\}$, as desired. But also, applying the iterative process to $\r= 010213$, which has $p$ at $\{1\}$, we get $011213$, which also has $q$ at position $\{1\}$, so we would stop.
But now we've mapped two different sequences to the same sequence! Under the recursive map, we would instead send $010223 \mapsto 011203$ and $010213 \mapsto 011213$. Because the recursion depth can exceed 1 in this case, the argument for the iterative algorithm described in Theorem \ref{thm:iter} does not apply. 

While 
we do not have 
a more direct combinatorial interpretation of this bijection, in Section \ref{sec:superduper} we give an alternative bijection which has a more straightforward description.

Here we include examples of how the recursive bijection works for this case.
Viewing the inversion sequences as underdiagonal lattice paths, we can visualize the steps of the recursive map as illustrated below.

\begin{example}
\begin{enumerate}[(a)]
  \item For  $\r = 010223$, the recursive map is computed as follows. \\ 
Let $p = \underline{0102}$ and $q = \underline{0112}$. We currently have $p$ at $\{ 1\}$; we want $q$ at $\{ 1\}$. \\ 
$\ph_{ = \{ 1\} } (010223)$: \\ 
\hphantom{...}$\mid$\,\,$\ph_{ \geq \{ 1\} } (010223) = 011223$ \\ 
\hphantom{...}$\mid$\,\,$q$ occurs at $\{ 1, 3\}$, so apply $\ph\inv_{ = \{ 1, 3\} }$ \\ 
\hphantom{...}$\mid$\,\,$\ph\inv_{ = \{ 1, 3\} } (011223)$: \\ 
\hphantom{...}$\mid$\,\,\hphantom{...}$\mid$\,\,$\ph\inv_{ \geq \{ 1, 3\} } (011223) = 010203$ \\ 
\hphantom{...}$\mid$\,\,\hphantom{...}$\mid$\,\,$p$ occurs at $\{ 1, 3\}$, so we are done. \\ 
\hphantom{...}$\mid$\,\,$010203$ \\ 
\hphantom{...}$\mid$\,\,$\ph_{ \geq \{ 1\} } (010203) = 011203$ \\ 
\hphantom{...}$\mid$\,\,$q$ occurs at $\{ 1\}$, so we are done. \\ 
$011203$ \\ 
We finally obtain $010223 \mapsto 011203$.  \\ 

The following pictures show the underdiagonal lattice paths of the original sequence, and the sequences after each application of $\ph_{\geq \{1\}}$.
\begin{figure}[H]
    \centering
    \,\hfill
     \begin{tikzpicture}[scale = 0.6] 
\draw[thick, green] (2.1, -.3 ) to[out=-30,in= 210] (5.9, -.3); 
\draw[green] (4,-1.3) node{$q$}; 
\filldraw[green, opacity=.3] (2, 0) -- (2, 0) -- (3, 0) -- (3,0); 
\filldraw[green, opacity=.3] (3, 0) -- (3, 2) -- (4, 2) -- (4,0); 
\filldraw[green, opacity=.3] (4, 0) -- (4, 2) -- (5, 2) -- (5,0); 
\filldraw[green, opacity=.3] (5, 0) -- (5, 3) -- (6, 3) -- (6,0); 
\filldraw[blue, opacity=.3] (0, 0) -- (0, 0) -- (1, 0) -- (1,0); 
\filldraw[blue, opacity=.3] (1, 0) -- (1, 1) -- (2, 1) -- (2,0); 
\filldraw[blue, opacity=.3] (2, 0) -- (2, 0) -- (3, 0) -- (3,0); 
\filldraw[blue, opacity=.3] (3, 0) -- (3, 2) -- (4, 2) -- (4,0); 
\draw[thick, blue] (0.1, -.3 ) to[out=-30,in= 210] (3.9, -.3); 
\draw[blue] (2,-1.3) node{$p$}; 
\draw (0,0) -- (1, 0); 
\draw (1,0) -- (1, 1) -- (2, 1); 
\draw (2,1) -- (2, 0) -- (3, 0); 
\draw (3,0) -- (3, 2) -- (4, 2); 
\draw (4,2) -- (4, 2) -- (5, 2); 
\draw (5,2) -- (5, 3) -- (6, 3); 
\end{tikzpicture}, \hfill
\begin{tikzpicture}[scale = .6] 
\draw[thick, green] (0.1, -.3 ) to[out=-30,in= 210] (3.9, -.3); 
\draw[green] (2,-1.3) node{$q$}; 
\filldraw[green, opacity=.3] (0, 0) -- (0, 0) -- (1, 0) -- (1,0); 
\filldraw[green, opacity=.3] (1, 0) -- (1, 1) -- (2, 1) -- (2,0); 
\filldraw[green, opacity=.3] (2, 0) -- (2, 1) -- (3, 1) -- (3,0); 
\filldraw[green, opacity=.3] (3, 0) -- (3, 2) -- (4, 2) -- (4,0); 
\draw[thick, green] (2.1, -.3 ) to[out=-30,in= 210] (5.9, -.3); 
\draw[green] (4,-1.3) node{$q$}; 
\filldraw[green, opacity=.3] (2, 0) -- (2, 1) -- (3, 1) -- (3,0); 
\filldraw[green, opacity=.3] (3, 0) -- (3, 2) -- (4, 2) -- (4,0); 
\filldraw[green, opacity=.3] (4, 0) -- (4, 2) -- (5, 2) -- (5,0); 
\filldraw[green, opacity=.3] (5, 0) -- (5, 3) -- (6, 3) -- (6,0); 
\draw[thick, dashed, blue] (0,0) -- (1, 0); 
\draw[thick, dashed, blue] (1,0) -- (1, 1) -- (2, 1); 
\draw[thick, dashed, blue] (2,1) -- (2, 1) -- (3, 1); 
\draw[thick, dashed, blue] (3,1) -- (3, 2) -- (4, 2); 
\draw[thick, dashed, blue] (4,2) -- (4, 2) -- (5, 2); 
\draw[thick, dashed, blue] (5,2) -- (5, 3) -- (6, 3); 
\end{tikzpicture}, \hfill
\begin{tikzpicture}[scale = .6] 
\draw[thick, green] (0.1, -.3 ) to[out=-30,in= 210] (3.9, -.3); 
\draw[green] (2,-1.3) node{$q$}; 
\filldraw[green, opacity=.3] (0, 0) -- (0, 0) -- (1, 0) -- (1,0); 
\filldraw[green, opacity=.3] (1, 0) -- (1, 1) -- (2, 1) -- (2,0); 
\filldraw[green, opacity=.3] (2, 0) -- (2, 1) -- (3, 1) -- (3,0); 
\filldraw[green, opacity=.3] (3, 0) -- (3, 2) -- (4, 2) -- (4,0); 
\draw[thick, dashed, blue] (0,0) -- (1, 0); 
\draw[thick, dashed, blue] (1,0) -- (1, 1) -- (2, 1); 
\draw[thick, dashed, blue] (2,1) -- (2, 1) -- (3, 1); 
\draw[thick, dashed, blue] (3,1) -- (3, 2) -- (4, 2); 
\draw[thick, dashed, blue] (4,2) -- (4, 0) -- (5, 0); 
\draw[thick, dashed, blue] (5,0) -- (5, 3) -- (6, 3); 
\end{tikzpicture}.\hfill \,
    \caption{Underdiagonal lattice path diagrams for 010223, 011223, 011203.}
\end{figure}
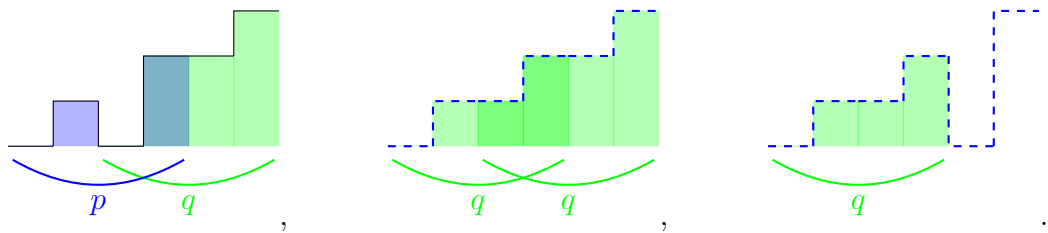

\item For $\r = 010213$, we currently have $p$ at $\{ 1\}$, and want $q$ at $\{ 1\}$. \\ 
$\ph_{ = \{ 1\} } (010213)$: \\ 
\hphantom{...}$\mid$\,\,$\ph_{ \geq \{ 1\} } (010213) = 011213$ \\ 
\hphantom{...}$\mid$\,\,$q$ occurs at $\{ 1\}$, so we are done. \\ 
$011213$ \\ 
We finally obtain $010213 \mapsto 011213$.  \\ 
 The following pictures show the underdiagonal lattice paths of the original sequence, and the sequence after the application of $\ph_{\geq \{1\}}$.

 \begin{figure}[H]
     \centering
     \begin{tikzpicture}[scale = 0.6] 
\filldraw[blue, opacity=.3] (0, 0) -- (0, 0) -- (1, 0) -- (1,0); 
\filldraw[blue, opacity=.3] (1, 0) -- (1, 1) -- (2, 1) -- (2,0); 
\filldraw[blue, opacity=.3] (2, 0) -- (2, 0) -- (3, 0) -- (3,0); 
\filldraw[blue, opacity=.3] (3, 0) -- (3, 2) -- (4, 2) -- (4,0); 
\draw[thick, blue] (0.1, -.3 ) to[out=-30,in= 210] (3.9, -.3); 
\draw[blue] (2,-1.3) node{$p$}; 
\draw (0,0) -- (1, 0); 
\draw (1,0) -- (1, 1) -- (2, 1); 
\draw (2,1) -- (2, 0) -- (3, 0); 
\draw (3,0) -- (3, 2) -- (4, 2); 
\draw (4,2) -- (4, 1) -- (5, 1); 
\draw (5,1) -- (5, 3) -- (6, 3); 
\end{tikzpicture}, \hspace{2cm}
\begin{tikzpicture}[scale = .6] 
\filldraw[blue, opacity=.3] (2, 0) -- (2, 1) -- (3, 1) -- (3,0); 
\filldraw[blue, opacity=.3] (3, 0) -- (3, 2) -- (4, 2) -- (4,0); 
\filldraw[blue, opacity=.3] (4, 0) -- (4, 1) -- (5, 1) -- (5,0); 
\filldraw[blue, opacity=.3] (5, 0) -- (5, 3) -- (6, 3) -- (6,0); 
\draw[thick, blue] (2.1, -.3 ) to[out=-30,in= 210] (5.9, -.3); 
\draw[blue] (4,-1.3) node{$p$}; 
\draw[thick, green] (0.1, -.3 ) to[out=-30,in= 210] (3.9, -.3); 
\draw[green] (2,-1.3) node{$q$}; 
\filldraw[green, opacity=.3] (0, 0) -- (0, 0) -- (1, 0) -- (1,0); 
\filldraw[green, opacity=.3] (1, 0) -- (1, 1) -- (2, 1) -- (2,0); 
\filldraw[green, opacity=.3] (2, 0) -- (2, 1) -- (3, 1) -- (3,0); 
\filldraw[green, opacity=.3] (3, 0) -- (3, 2) -- (4, 2) -- (4,0); 
\draw[thick, dashed, blue] (0,0) -- (1, 0); 
\draw[thick, dashed, blue] (1,0) -- (1, 1) -- (2, 1); 
\draw[thick, dashed, blue] (2,1) -- (2, 1) -- (3, 1); 
\draw[thick, dashed, blue] (3,1) -- (3, 2) -- (4, 2); 
\draw[thick, dashed, blue] (4,2) -- (4, 1) -- (5, 1); 
\draw[thick, dashed, blue] (5,1) -- (5, 3) -- (6, 3); 
\end{tikzpicture}.
     \caption{Underdiagonal lattice path diagrams for 010213, 011213.}
 \end{figure}
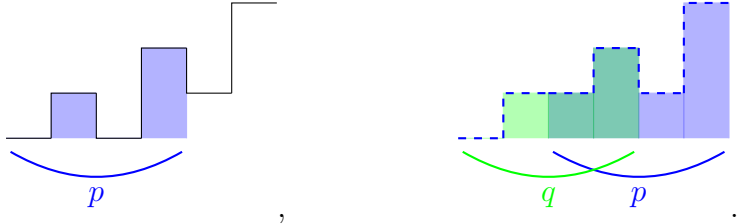
\end{enumerate}
\end{example}

For a sequence which requires larger recursion depth, the steps can become more complicated, as seen in the next example.

\begin{example} \label{ex:01022304}

Suppose $p = \underline{0102}$, $q = \underline{0112}$. 
Let $\r = 01022304$. \\
We currently have $p$ at $\{ 1\}$; we want $q$ at $\{ 1\}$. \\ 
$\ph_{ = \{ 1\} } (01022304)$: \\ 
\hphantom{...}$\mid$\,\,$\ph_{ \geq \{ 1\} } (01022304) = 01122304$ \\ 
\hphantom{...}$\mid$\,\,$q$ occurs at $\{ 1, 3\}$, so apply $\ph\inv_{ = \{ 1, 3\} }$ \\ 
\hphantom{...}$\mid$\,\,$\ph\inv_{ = \{ 1, 3\} } (01122304)$: \\ 
\hphantom{...}$\mid$\,\,\hphantom{...}$\mid$\,\,$\ph\inv_{ \geq \{ 1, 3\} } (01122304) = 01020304$ \\ 
\hphantom{...}$\mid$\,\,\hphantom{...}$\mid$\,\,$p$ occurs at $\{ 1, 3, 5\}$, so apply $\ph_{ = \{ 1, 3, 5\} }$ \\ 
\hphantom{...}$\mid$\,\,\hphantom{...}$\mid$\,\,$\ph_{ = \{ 1, 3, 5\} } (01020304)$: \\ 
\hphantom{...}$\mid$\,\,\hphantom{...}$\mid$\,\,\hphantom{...}$\mid$\,\,$\ph_{ \geq \{ 1, 3, 5\} } (01020304) = 01122334$ \\ 
\hphantom{...}$\mid$\,\,\hphantom{...}$\mid$\,\,\hphantom{...}$\mid$\,\,$q$ occurs at $\{ 1, 3, 5\}$, so we are done. \\ 
\hphantom{...}$\mid$\,\,\hphantom{...}$\mid$\,\,$01122334$ \\ 
\hphantom{...}$\mid$\,\,\hphantom{...}$\mid$\,\,$\ph\inv_{ \geq \{ 1, 3\} } (01122334) = 01020334$ \\ 
\hphantom{...}$\mid$\,\,\hphantom{...}$\mid$\,\,$p$ occurs at $\{ 1, 3\}$, so we are done. \\ 
\hphantom{...}$\mid$\,\,$01020334$ \\ 
\hphantom{...}$\mid$\,\,$\ph_{ \geq \{ 1\} } (01020334) = 01120334$ \\ 
\hphantom{...}$\mid$\,\,$q$ occurs at $\{ 1, 5\}$, so apply $\ph\inv_{ = \{ 1, 5\} }$ \\ 
\hphantom{...}$\mid$\,\,$\ph\inv_{ = \{ 1, 5\} } (01120334)$: \\ 
\hphantom{...}$\mid$\,\,\hphantom{...}$\mid$\,\,$\ph\inv_{ \geq \{ 1, 5\} } (01120334) = 01020304$ \\ 
\hphantom{...}$\mid$\,\,\hphantom{...}$\mid$\,\,$p$ occurs at $\{ 1, 3, 5\}$, so apply $\ph_{ = \{ 1, 3, 5\} }$ \\ 
\hphantom{...}$\mid$\,\,\hphantom{...}$\mid$\,\,$\ph_{ = \{ 1, 3, 5\} } (01020304)$: \\ 
\hphantom{...}$\mid$\,\,\hphantom{...}$\mid$\,\,\hphantom{...}$\mid$\,\,$\ph_{ \geq \{ 1, 3, 5\} } (01020304) = 01122334$ \\ 
\hphantom{...}$\mid$\,\,\hphantom{...}$\mid$\,\,\hphantom{...}$\mid$\,\,$q$ occurs at $\{ 1, 3, 5\}$, so we are done. \\ 
\hphantom{...}$\mid$\,\,\hphantom{...}$\mid$\,\,$01122334$ \\ 
\hphantom{...}$\mid$\,\,\hphantom{...}$\mid$\,\,$\ph\inv_{ \geq \{ 1, 5\} } (01122334) = 01022324$ \\ 
\hphantom{...}$\mid$\,\,\hphantom{...}$\mid$\,\,$p$ occurs at $\{ 1, 5\}$, so we are done. \\ 
\hphantom{...}$\mid$\,\,$01022324$ \\ 
\hphantom{...}$\mid$\,\,$\ph_{ \geq \{ 1\} } (01022324) = 01122324$ \\ 
\hphantom{...}$\mid$\,\,$q$ occurs at $\{ 1, 3\}$, so apply $\ph\inv_{ = \{ 1, 3\} }$ \\ 
\hphantom{...}$\mid$\,\,$\ph\inv_{ = \{ 1, 3\} } (01122324)$: \\ 
\hphantom{...}$\mid$\,\,\hphantom{...}$\mid$\,\,$\ph\inv_{ \geq \{ 1, 3\} } (01122324) = 01020324$ \\ 
\hphantom{...}$\mid$\,\,\hphantom{...}$\mid$\,\,$p$ occurs at $\{ 1, 3\}$, so we are done. \\ 
\hphantom{...}$\mid$\,\,$01020324$ \\ 
\hphantom{...}$\mid$\,\,$\ph_{ \geq \{ 1\} } (01020324) = 01120324$ \\ 
\hphantom{...}$\mid$\,\,$q$ occurs at $\{ 1\}$, so we are done. \\ 
$01120324$ \\ 
We finally obtain $01022304 \mapsto 01120324$.

The following pictures show the underdiagonal lattice paths of the original sequence $01022304$, and the resulting sequences after each application of $\ph_{\geq \{0\}}$.

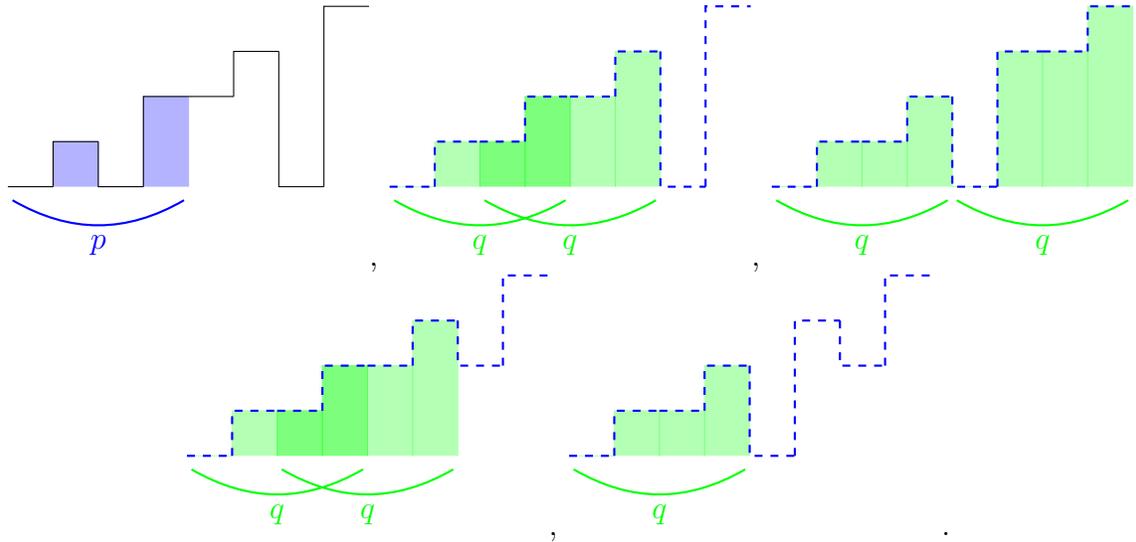
\begin{figure}[H]
    \centering
\begin{tikzpicture}[scale = 0.6] 
\filldraw[blue, opacity=.3] (0, 0) -- (0, 0) -- (1, 0) -- (1,0); 
\filldraw[blue, opacity=.3] (1, 0) -- (1, 1) -- (2, 1) -- (2,0); 
\filldraw[blue, opacity=.3] (2, 0) -- (2, 0) -- (3, 0) -- (3,0); 
\filldraw[blue, opacity=.3] (3, 0) -- (3, 2) -- (4, 2) -- (4,0); 
\draw[thick, blue] (0.1, -.3 ) to[out=-30,in= 210] (3.9, -.3); 
\draw[blue] (2,-1.3) node{$p$}; 
\draw (0,0) -- (1, 0); 
\draw (1,0) -- (1, 1) -- (2, 1); 
\draw (2,1) -- (2, 0) -- (3, 0); 
\draw (3,0) -- (3, 2) -- (4, 2); 
\draw (4,2) -- (4, 2) -- (5, 2); 
\draw (5,2) -- (5, 3) -- (6, 3); 
\draw (6,3) -- (6, 0) -- (7, 0); 
\draw (7,0) -- (7, 4) -- (8, 4); 
\end{tikzpicture},
\begin{tikzpicture}[scale = .6] 
\draw[thick, green] (0.1, -.3 ) to[out=-30,in= 210] (3.9, -.3); 
\draw[green] (2,-1.3) node{$q$}; 
\filldraw[green, opacity=.3] (0, 0) -- (0, 0) -- (1, 0) -- (1,0); 
\filldraw[green, opacity=.3] (1, 0) -- (1, 1) -- (2, 1) -- (2,0); 
\filldraw[green, opacity=.3] (2, 0) -- (2, 1) -- (3, 1) -- (3,0); 
\filldraw[green, opacity=.3] (3, 0) -- (3, 2) -- (4, 2) -- (4,0); 
\draw[thick, green] (2.1, -.3 ) to[out=-30,in= 210] (5.9, -.3); 
\draw[green] (4,-1.3) node{$q$}; 
\filldraw[green, opacity=.3] (2, 0) -- (2, 1) -- (3, 1) -- (3,0); 
\filldraw[green, opacity=.3] (3, 0) -- (3, 2) -- (4, 2) -- (4,0); 
\filldraw[green, opacity=.3] (4, 0) -- (4, 2) -- (5, 2) -- (5,0); 
\filldraw[green, opacity=.3] (5, 0) -- (5, 3) -- (6, 3) -- (6,0); 
\draw[thick, dashed, blue] (0,0) -- (1, 0); 
\draw[thick, dashed, blue] (1,0) -- (1, 1) -- (2, 1); 
\draw[thick, dashed, blue] (2,1) -- (2, 1) -- (3, 1); 
\draw[thick, dashed, blue] (3,1) -- (3, 2) -- (4, 2); 
\draw[thick, dashed, blue] (4,2) -- (4, 2) -- (5, 2); 
\draw[thick, dashed, blue] (5,2) -- (5, 3) -- (6, 3); 
\draw[thick, dashed, blue] (6,3) -- (6, 0) -- (7, 0); 
\draw[thick, dashed, blue] (7,0) -- (7, 4) -- (8, 4); 
\end{tikzpicture},
\begin{tikzpicture}[scale = .6] 
\draw[thick, green] (0.1, -.3 ) to[out=-30,in= 210] (3.9, -.3); 
\draw[green] (2,-1.3) node{$q$}; 
\filldraw[green, opacity=.3] (0, 0) -- (0, 0) -- (1, 0) -- (1,0); 
\filldraw[green, opacity=.3] (1, 0) -- (1, 1) -- (2, 1) -- (2,0); 
\filldraw[green, opacity=.3] (2, 0) -- (2, 1) -- (3, 1) -- (3,0); 
\filldraw[green, opacity=.3] (3, 0) -- (3, 2) -- (4, 2) -- (4,0); 
\draw[thick, green] (4.1, -.3 ) to[out=-30,in= 210] (7.9, -.3); 
\draw[green] (6,-1.3) node{$q$}; 
\filldraw[green, opacity=.3] (4, 0) -- (4, 0) -- (5, 0) -- (5,0); 
\filldraw[green, opacity=.3] (5, 0) -- (5, 3) -- (6, 3) -- (6,0); 
\filldraw[green, opacity=.3] (6, 0) -- (6, 3) -- (7, 3) -- (7,0); 
\filldraw[green, opacity=.3] (7, 0) -- (7, 4) -- (8, 4) -- (8,0); 
\draw[thick, dashed, blue] (0,0) -- (1, 0); 
\draw[thick, dashed, blue] (1,0) -- (1, 1) -- (2, 1); 
\draw[thick, dashed, blue] (2,1) -- (2, 1) -- (3, 1); 
\draw[thick, dashed, blue] (3,1) -- (3, 2) -- (4, 2); 
\draw[thick, dashed, blue] (4,2) -- (4, 0) -- (5, 0); 
\draw[thick, dashed, blue] (5,0) -- (5, 3) -- (6, 3); 
\draw[thick, dashed, blue] (6,3) -- (6, 3) -- (7, 3); 
\draw[thick, dashed, blue] (7,3) -- (7, 4) -- (8, 4); 
\end{tikzpicture}, \\

\begin{tikzpicture}[scale = .6] 
\draw[thick, green] (0.1, -.3 ) to[out=-30,in= 210] (3.9, -.3); 
\draw[green] (2,-1.3) node{$q$}; 
\filldraw[green, opacity=.3] (0, 0) -- (0, 0) -- (1, 0) -- (1,0); 
\filldraw[green, opacity=.3] (1, 0) -- (1, 1) -- (2, 1) -- (2,0); 
\filldraw[green, opacity=.3] (2, 0) -- (2, 1) -- (3, 1) -- (3,0); 
\filldraw[green, opacity=.3] (3, 0) -- (3, 2) -- (4, 2) -- (4,0); 
\draw[thick, green] (2.1, -.3 ) to[out=-30,in= 210] (5.9, -.3); 
\draw[green] (4,-1.3) node{$q$}; 
\filldraw[green, opacity=.3] (2, 0) -- (2, 1) -- (3, 1) -- (3,0); 
\filldraw[green, opacity=.3] (3, 0) -- (3, 2) -- (4, 2) -- (4,0); 
\filldraw[green, opacity=.3] (4, 0) -- (4, 2) -- (5, 2) -- (5,0); 
\filldraw[green, opacity=.3] (5, 0) -- (5, 3) -- (6, 3) -- (6,0); 
\draw[thick, dashed, blue] (0,0) -- (1, 0); 
\draw[thick, dashed, blue] (1,0) -- (1, 1) -- (2, 1); 
\draw[thick, dashed, blue] (2,1) -- (2, 1) -- (3, 1); 
\draw[thick, dashed, blue] (3,1) -- (3, 2) -- (4, 2); 
\draw[thick, dashed, blue] (4,2) -- (4, 2) -- (5, 2); 
\draw[thick, dashed, blue] (5,2) -- (5, 3) -- (6, 3); 
\draw[thick, dashed, blue] (6,3) -- (6, 2) -- (7, 2); 
\draw[thick, dashed, blue] (7,2) -- (7, 4) -- (8, 4); 
\end{tikzpicture},
\begin{tikzpicture}[scale = .6] 
\draw[thick, green] (0.1, -.3 ) to[out=-30,in= 210] (3.9, -.3); 
\draw[green] (2,-1.3) node{$q$}; 
\filldraw[green, opacity=.3] (0, 0) -- (0, 0) -- (1, 0) -- (1,0); 
\filldraw[green, opacity=.3] (1, 0) -- (1, 1) -- (2, 1) -- (2,0); 
\filldraw[green, opacity=.3] (2, 0) -- (2, 1) -- (3, 1) -- (3,0); 
\filldraw[green, opacity=.3] (3, 0) -- (3, 2) -- (4, 2) -- (4,0); 
\draw[thick, dashed, blue] (0,0) -- (1, 0); 
\draw[thick, dashed, blue] (1,0) -- (1, 1) -- (2, 1); 
\draw[thick, dashed, blue] (2,1) -- (2, 1) -- (3, 1); 
\draw[thick, dashed, blue] (3,1) -- (3, 2) -- (4, 2); 
\draw[thick, dashed, blue] (4,2) -- (4, 0) -- (5, 0); 
\draw[thick, dashed, blue] (5,0) -- (5, 3) -- (6, 3); 
\draw[thick, dashed, blue] (6,3) -- (6, 2) -- (7, 2); 
\draw[thick, dashed, blue] (7,2) -- (7, 4) -- (8, 4); 
\end{tikzpicture} .
    \caption{Underdiagonal lattice path diagrams for 01022304, 01122304, 01120334, 01122324, 01120324.}
\end{figure}

\end{example}

Remark: {The sequences for which the number of required applications of $\ph_\geq$ and $\ph_\geq\inv$ is maximized are not unique. For example, when $n = 8$, the following two computations both result in 19 function calls: $01120314 \mapsto 01021304$, $01130537 \mapsto 01031517$.
}

For $01120314 \mapsto 01021304$, we can see what looks like a direct relationship between the input and output. 
A larger example can be found in Appendix \ref{extendedexample}. This class of examples motivated the development of the algorithm in the following section.

\subsection{Reciprocal relations}\label{sec:superduper}

While the direct description of the previous bijection is not straightforward to deduce, we can actually produce a stronger result. Recall the following definition:

\begin{definition}
    Let $p$ and $q$ be consecutive patterns. We say that $p$ and $q$ are \textit{reciprocal} if for all pairs $S,T \ci [n]$, 
$$
|\{\r \in I_n : \Em(p,\r) = S, \Em(q,\r) = T\}| = |\{\r\in I_n : \Em(p,\r) = T, \Em(q,\r) = S\}|.
$$
In other words, we can swap the occurrences of $p$ and $q$ in an inversion sequence. 
When $p$ and $q$ satisfy this relation, we denote this by $p \recip
 q$. 
\end{definition}

It turns out that the consecutive patterns $\underline{0102}$ and $\underline{0112}$ are reciprocal as well.
Algorithm \ref{alg:sd} takes a sequence $\seq$ and switches all instances of $p = \underline{0102}$ and $q = \underline{0112}$ in $\seq$.

\begin{algorithm}
\label{sdalgorithm}
\caption{Bijection to show that $p = \underline{0102}$ and $q = \underline{0112}$ are reciprocal}\label{alg:sd}
\begin{algorithmic}

\Require sequence $\seq$ with occurrences of $p$ and $q$ at positions $S$ and $T$, respectively.

\Ensure sequence with occurrences of $q$ and $p$ at positions $S$ and $T$, respectively.

\State $\Em(p) \gets $ indices of occurrences of $p$ in $\seq$ \;

\State $\Em(q) \gets $ indices of occurrences of $q$ in $\seq$ \;

\State $\last \gets$ null \;

\For{each position $i$ in $\seq$}

\State $\newEm(p) \gets$ current indices of occurrences of $p$ in $\seq$\;

\State  $\newEm(q) \gets$ current indices of occurrences of $q$ in $\seq$ \;
        
    \If{$i-2 \in \Em(p)$} 
    
    \Comment{The $i$th entry is the third position of an original occurrence of $p$ in $\seq$.}

\State $\last \gets $ \seq[i]

\Comment{Store whatever is currently in position $i$ into $\last$.}

\State $\seq[i] \gets \seq[i-1]$

\Comment{Copy the entry directly to the left into position $i$}

\Comment{(this changes the original $p$ to a $q$).}

\ElsIf{ $i-2 \in \Em(q)$}

\Comment{The $i$th entry is the third position of an original occurrence of $q$ in $\seq$.}

\State $\last \gets $ \seq[i]

\Comment{Store whatever is currently in position $i$ into $\last$.}

\State $\seq[i] \gets \seq[i-2]$

\Comment{Copy whatever digit is two to the left into position $i$} 

\Comment{(this changes the original $q$ to a $p$).}

\Comment{Now if the $i$th entry of $\seq$ is the third position of a new (unwanted) occurrence of $p$ or $q$ in $\seq$...}

\ElsIf{$i-2 \in \newEm(p)$}

\State $\texttt{temp} \gets \seq[i]$

\State $\seq[i] \gets \last$
\State $\last \gets \texttt{temp}$

\Comment{Copy whatever is in $\last$ into position $i$, and store the old digit in $\last$.}

\ElsIf{$i-2 \in \newEm(q)$}

\State $\texttt{temp} \gets \seq[i]$

\State $\seq[i] \gets \last$
\State $\last \gets \texttt{temp}$

\EndIf
\EndFor\\
\Return \texttt{seq}

\end{algorithmic}
\end{algorithm}

After collecting several upcoming results about the algorithm, we see that applying the algorithm twice can be visualized by the following diagrams for each case:

\begin{figure}[H]
    \centering

\begin{minipage}{.32\linewidth}
\begin{center}
    
    if $\r$ has an occurrence of \\ $p$ at $i$:\\
\vspace{3mm}
\fbox{
\begin{tikzpicture}[scale=1]
\draw
(0,0) node (e1){$\r_i$}
(1,0) node(e2){$\r_{i+1}$}
(2,0) node(e3){$\r_{i+2}$}
(3,0) node(e4){$\r_{i+3}$}

(0,-1.5) node(e'1){$\r_i'$}
(1,-1.5) node(e'2){$\r_{i+1}'$}
(2,-1.5) node(e'3){$\r_{i+2}'$}
(3,-1.5) node(e'4){$\r_{i+3}'$}

(0,-3) node(e''1){$\r_i''$}
(1,-3) node(e''2){$\r_{i+1}''$}
(2,-3) node(e''3){$\r_{i+2}''$}
(3,-3) node(e''4){$\r_{i+3}''$}

;

\path[line width = 1mm, dotted, magenta]
(e1) edge[bend right=40] node [right] {} (e3)
(e''2) edge node [right] {} (e'3)

;

\path[line width = 1mm, ->,>=stealth, cyan]
(e''1) edge[bend left=40] node [right] {} (e''3)
(e2) edge node [right] {} (e'3)
;

\end{tikzpicture}
}
   
    \end{center}
\end{minipage}\hfill
\begin{minipage}{.32\linewidth}
    \begin{center}

if $\r$ has an occurrence of \\ $q$ at $i$:\\
\vspace{3mm}
\fbox{
\begin{tikzpicture}[scale=1]
\draw
(0,0) node (e1){$\r_i$}
(1,0) node(e2){$\r_{i+1}$}
(2,0) node(e3){$\r_{i+2}$}
(3,0) node(e4){$\r_{i+3}$}

(0,-1.5) node(e'1){$\r_i'$}
(1,-1.5) node(e'2){$\r_{i+1}'$}
(2,-1.5) node(e'3){$\r_{i+2}'$}
(3,-1.5) node(e'4){$\r_{i+3}'$}

(0,-3) node(e''1){$\r_i''$}
(1,-3) node(e''2){$\r_{i+1}''$}
(2,-3) node(e''3){$\r_{i+2}''$}
(3,-3) node(e''4){$\r_{i+3}''$}

;

\path[line width = 1mm, dotted, magenta]
(e3) edge node [right] {} (e'2)

(e'1) edge[bend right=40] node [right] {} (e'3)
;

\path[line width = 1mm, ->,>=stealth, cyan]
(e'1) edge[bend left=40] node [right] {} (e'3)
(e'2) edge node [right] {} (e''3)
;

\end{tikzpicture}
}
   
    \end{center}
\end{minipage}\hfill
\begin{minipage}{.32\linewidth}
    \begin{center}
if an ``extra'' occurrence of\\ $p$ appears at $i$:\\
\vspace{3mm}
\fbox{
\begin{tikzpicture}[scale=1]
\draw
(0,0) node (e1){$\r_i$}
(1,0) node(e2){$\r_{i+1}$}
(2,0) node(e3){$\r_{i+2}$}
(3,0) node(e4){$\r_{i+3}$}

(0,-1.5) node(e'1){$\r_i'$}
(1,-1.5) node(e'2){$\r_{i+1}'$}
(2,-1.5) node(e'3){$\r_{i+2}'$}
(3,-1.5) node(e'4){$\r_{i+3}'$}

(0,-3) node(e''1){$\r_i''$}
(1,-3) node(e''2){$\r_{i+1}''$}
(2,-3) node(e''3){$\r_{i+2}''$}
(3,-3) node(e''4){$\r_{i+3}''$}

;

\path[line width = 1mm, dotted, magenta]
(e'1) edge node [right] {} (e3)
(e''1) edge node [right] {} (e'3)

;

\path[line width = 1mm, ->,>=stealth, cyan]
(e1) edge node [right] {} (e'3)
(e'1) edge node [right] {} (e''3)
;

\end{tikzpicture}
   }
    \end{center}

\end{minipage}
    \caption{Diagrams of replacements made by Algorithm \ref{alg:sd} in each case of an occurrence at $i$.}
\end{figure}
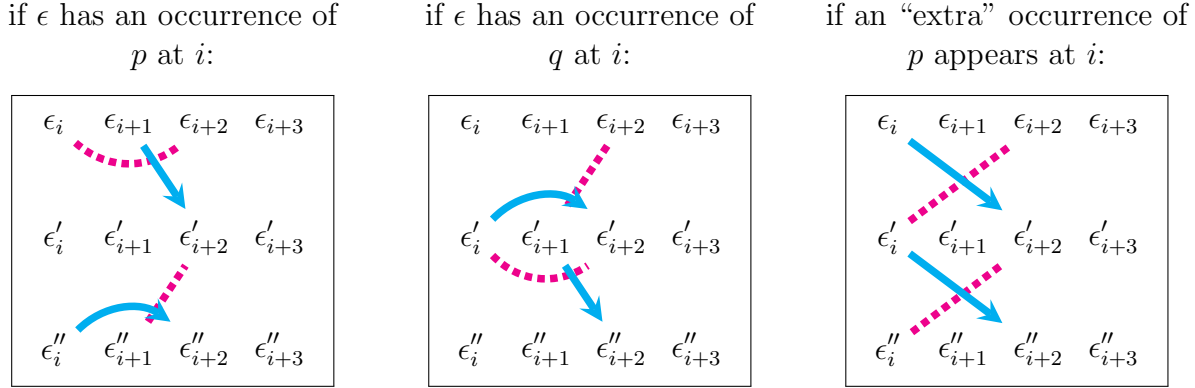

The dotted edges in the above diagrams show equalities between elements that cause a replacement to be made (additional required inequalities are not indicated). The solid edges show how entries in the sequence are copied to other positions in the sequence. So each edge denotes an equality of sequence entries.

\begin{samepage}
    
\begin{example}
\,
\begin{figure}[H]
    \centering
    \begin{tikzpicture}[scale=1]
\draw
(-1,0) node {$\r$:}
(0,0) node (e1){0}
(1,0) node(e2){1}
(2,0) node(e3){1}
(3,0) node(e4){2}
(4,0) node(e5){2}
(5,0) node(e6){3}
(6,0) node(e7){0}
(7,0) node(e8){4}
(8,0) node(e9) {2}
(9,0) node(e10){5}
(10,0) node(e11){2}
(11,0) node (e12) {6}

(-1,-1.4) node {$\r'$:}
(0,-1.5) node(e'1){0}
(1,-1.5) node(e'2){1}
(2,-1.5) node(e'3){0}
(3,-1.5) node(e'4){2}
(4,-1.5) node(e'5){0}
(5,-1.5) node(e'6){3}
(6,-1.5) node(e'7){2}
(7,-1.5) node(e'8){4}
(8,-1.5) node(e'9) {0}
(9,-1.5) node(e'10){5}
(10,-1.5) node(e'11){5}
(11,-1.5) node (e'12) {6}

(-1,-2.9) node {$\r''$:}
(0,-3) node (e''1){0}
(1,-3) node(e''2){1}
(2,-3) node(e''3){1}
(3,-3) node(e''4){2}
(4,-3) node(e''5){2}
(5,-3) node(e''6){3}
(6,-3) node(e''7){0}
(7,-3) node(e''8){4}
(8,-3) node(e''9) {2}
(9,-3) node(e''10){5}
(10,-3) node(e''11){2}
(11,-3) node (e''12) {6}
;

\path[line width = .7mm, dotted, magenta]
(e3) edge node [right] {} (e'2)
(e'4) edge node [right] {} (e5)
(e'5) edge node [right] {} (e7)
(e'7) edge node [right] {} (e9)
(e9) edge[bend right=40] node [right] {} (e11)

(e'1) edge[bend right=40] node [right] {} (e'3)
(e'3) edge[bend right=40] node [right] {} (e'5)
(e''5) edge node [right] {} (e'7)
(e''7) edge node [right] {} (e'9)
(e''10) edge node [right] {} (e'11)
;

\path[line width = .7mm, ->,>=stealth, cyan]
(e'1) edge[bend left=40] node [right] {} (e'3)
(e'3) edge[bend left=40] node [right] {} (e'5)
(e5) edge node [right] {} (e'7)
(e7) edge node [right] {} (e'9)
(e10) edge node [right] {} (e'11)

(e'2) edge node [right] {} (e''3)
(e'4) edge node [right] {} (e''5)
(e'5) edge node [right] {} (e''7)
(e'7) edge node [right] {} (e''9)
(e''9) edge[bend left=40] node [right] {} (e''11)
;

\end{tikzpicture}
    \caption{Diagram of replacements made when applying Algorithm \ref{alg:sd} twice to the sequence $\r = 011223042526$.}
\end{figure}
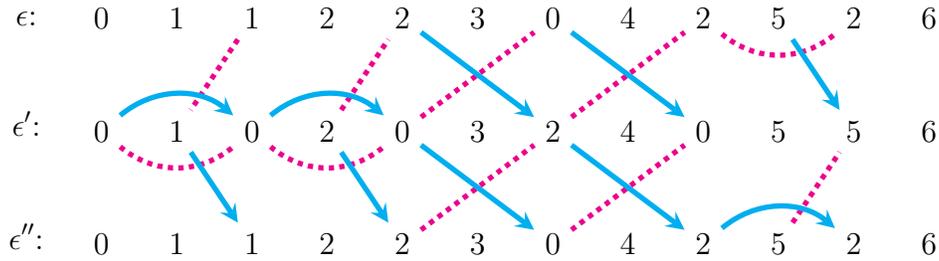
\end{example}

\end{samepage}

The following results will show that this algorithm gives a bijection between $\{ \r\in I_n : \Em(p,\r) = S \text{ and } \Em(q,\r) = T\}$ and $\{ \r\in I_n : \Em(q,\r) = S \text{ and } \Em(p,\r) = T \}$ for all $S,T \ci [n]$. First, we show that given an inversion sequence as the input, the algorithm will always output an inversion sequence.

\begin{prop}
\label{makesinvseq}
Let $\r = \r_1 \r_2 \cdots \r_n$ be an inversion sequence, and let $\r'$ be the resulting sequence after applying Algorithm \ref{alg:sd} to $\r$. Then $\r'$ is also an inversion sequence.
\end{prop}

\begin{proof}
    Since $\r$ is an inversion sequence, we know that $0 \leq \r_i < i$ for all $i$. If Algorithm \ref{alg:sd} replaces any entries of $\r$, the entries will be replaced with entries from earlier in the sequence. That is, if $\r_j$ is replaced, then $\r_j' = \r_i$ for some $i<j$. Then
    $$
    0 \leq \r_j' = \r_i<i <j,
    $$
    so $0 \leq \r_j'<j$ for all $j$. 
    Also, if $\r_j$ is not replaced, we have $0 \leq \r_j' = \r_j <j$.
    Thus $\r'$ is also an inversion sequence.
\end{proof}

\begin{lemma}\label{lem0} Let $\r$ be an initial sequence to which the algorithm is applied. Suppose $i$ is an occurrence of $p$ or $q$ in the sequence at any step in the algorithm. Then among the set of entries $
\{\r_i, \r_{i+1}, \r_{i+2}, \r_{i+3}\}$, only $\r_i$ or $\r_{i+2}$ could have been changed from the original sequence. \end{lemma}

\begin{proof}
The algorithm only makes replacements at the third position of either occurrences of $p$ or $q$ that appeared in the original sequence (even if they are not currently occurrences when the algorithm reaches that position) or occurrences of $p$ or $q$ that are created as entries are replaced.

The consecutive patterns $\underline{0102}$ and $\underline{0112}$ can only overlap themselves or each other in one or two entries. In particular, the third entry of one pattern occurrence could either not be part of another pattern occurrence, or be the first entry of another pattern occurrence. Since these are the only entries in the sequence which are changed, and these are the only cases for where they can appear in a pattern occurrence, the result follows.
\end{proof}

We will refer to the occurrences of $p$ or $q$ that appeared in the original sequence as \textit{original} $p$ or $q$.  A position satisfying this is considered an original $p$ or $q$ even if it is not an occurrence of $p$ or $q$ in the sequence at some point during the algorithm. We say an occurrence is a \textit{transient} $p$ or $q$ if it is a non-original occurrence of $p$ or $q$ that appears during the algorithm as a result of replacements. 

That is, if $\r$ is the initial sequence and $\r'$ is the sequence resulting from applying the algorithm to $\r$, we say there is
\begin{itemize}
    \item an \textit{original} occurrence of $p$ at $i$ if $\r_i = \r_{i+2}$, and $\r_i<\r_{i+1}<\r_{i+3}$,
    \item an \textit{original} occurrence of $q$ at $i$ if $\r_i <\r_{i+1}= \r_{i+2} <\r_{i+3}$,
    \item  a \textit{transient} occurrence of $p$ at $i$ if $i$ is not an original occurrence of $p$, but $\r_i' = \r_{i+2}$, and $\r_i'<\r_{i+1}<\r_{i+3}$, and
    \item a \textit{transient} occurrence of $q$ at $i$  if $i$ is not an original occurrence of $q$, but $\r_i' <\r_{i+1}= \r_{i+2} <\r_{i+3}$.
\end{itemize} 

Note: It follows from these definitions that if there is a transient $p$ or $q$ in position $i$, then $\r_i' \neq \r_i$.
We will also show that the non-original occurrences of $p$ or $q$ are indeed ``transient'' as these occurrences are removed later in the algorithm.

\begin{lemma}\label{lem1} Let $\r$ be an inversion sequence to which the algorithm is applied, and let $\r'$ be the resulting sequence. If an entry $\r_i$ in the sequence is replaced during 
the algorithm, 
\begin{enumerate}[(i)]
    \item $\r_i < \r_{i+1}$, and
    \item the replacement entry $\r_i'$ is less than $\r_{i+1}$.
\end{enumerate} \end{lemma}

\begin{proof}
(i): The third entry in an occurrence of either pattern will be less than the fourth entry. The algorithm only makes replacements at the third position of either original occurrences of $p$ or $q$, or transient occurrences of $p$ or $q$. 
In the case of an original occurrence of $p$ or $q$ at index ${i-2}$,
we have $\r_{i} < \r_{i+1}$.
In the case of a transient occurrence of $p$ or $q$ at index ${i-2}$, we have $\r_{i} < \r_{i+1}$.

(ii): The first time that an entry in the sequence will be replaced is in the first original occurrence of $p$ or $q$. Suppose the first such occurrence is at position $j$. Since no replacements have been made so far, we know the original $p$ or $q$ is currently unchanged. So $\r_j < \r_{j+3}$, and $\r_{j+1} < \r_{j+3}$. The entry in position $j$ or ${j+1}$ will replace the entry in position $j+2$. So the result holds for the position of the first (leftmost) replacement in the algorithm. 

We will proceed by induction. Suppose the result holds for all $k$ such that $0 \leq k < i$.

The algorithm will only replace an entry if it is in the third position of either an original occurrence of $p$ or $q$, or a transient occurrence of $p$ or $q$ in the current sequence.

Suppose $i$ is the third position of an original $p$ or $q$.
Either the original $p$ or $q$ in that position is completely unchanged, or the first digit was changed, by Lemma \ref{lem0}. 

If the original $p$ or $q$ is completely unchanged, then we have that $\r_{i-2}$ and $\r_{i-1}$ are both less than $ \r_{i+1}$, and the value of one of them will replace $\r_i$. In either case, the replaced digit will still be less than $\r_{i+1}$.

If the first entry of the original $p$ or $q$ was changed, then by the inductive hypothesis, we know that 
$\r_{i-2}'<\r_{i-1}$. 
Also, at the point in the algorithm when the first entry $\r_{i-2}$ was replaced, the second, third and fourth entries of the occurrence were all still unchanged. So the second entry is still less than the fourth, and the (replaced) first entry is less than the second, so the first entry is also less than the fourth. Then when the third digit is changed by replacing it with the first or second entry, the replacement is still less than the entry directly to its right.

Otherwise suppose that $i$ is the third position of a transient occurrence of $p$ or $q$. In this case, whatever is held in the variable $\last$ will replace the entry in position $i$. 

Since $i-2$ is a transient occurrence of $p$ or $q$, we know by Lemma \ref{lem0} that this occurrence must have resulted from a replacement in position $i-2$. 

Then $\r_{i-2}$ is now stored in $\last$.
By Lemma \ref{lem1}(i), since $\r_{i-2}$ was replaced, we have that 
$\r_{i-2} < \r_{i-1}$.
That is, the entry stored in $\last$ is less than $\r_{i-1}$.

Since $i-2$ is a transient occurrence of $p$ or $q$, it follows that 
$\r_{i-1} <\r_{i+1}$. So, the entry stored in $\last$ is less than $\r_{i+1}$. Since $\r_i$ is replaced with whatever entry is in $\last$, the result holds in this case as well.
\end{proof}

\begin{cor}
\label{lem1cor} In Algorithm \ref{alg:sd}, it is not possible to generate a transient $q$.
\end{cor}

\begin{proof}
Let $\r$ be the initial sequence, and let $\r'$ be the sequence resulting from applying the algorithm.
    Suppose to the contrary that a transient occurrence of $q$ appears at position $i$, so  $i$ is not an original occurrence of $q$, but $\r_i' <\r_{i+1}= \r_{i+2} <\r_{i+3}$. 
    Since $i$ is not an original occurrence of $q$, we know that the following must not all be true:
    $$
     \r_i < \r_{i+1} = \r_{i+2} < \r_{i+3}.
     $$
     It follows that the inequality that did not hold is $\r_i < \r_{i+1}$. This implies $\r_i' \neq \r_i$, so there must have been a replacement at position $i$. However, since the inequality $\r_i < \r_{i+1}$ does not hold, this contradicts Lemma \ref{lem1}.
\end{proof}

Remark: In light of the previous corollary, we can shorten the pseudocode for the algorithm to remove the conditional case for the transient $q$, since it will never be used.

The following result will also be useful.

\begin{cor}\label{cor:reversal}
Let $\r = \r_1 \r_2 \cdots \r_n$ be the reversal of an inversion sequence, and let $\r'$ be the resulting sequence after applying Algorithm \ref{alg:sd} to $\r$. Then $\r'$ is also the reversal of an inversion sequence.
\end{cor}

\begin{proof}
    A sequence $\r=  \r_1 \r_2 \cdots \r_n$ is the reversal of an inversion sequence if and only if $0 \leq \r_i \leq n-i$ for all $i$. As in the proof of Proposition \ref{makesinvseq}, if $\r_j$ is replaced, then $\r_j' = \r_i$ for some $i<j$. By Lemma \ref{lem1}(ii), we know $\r_j' < \r_{j+1}$. Then,
    $$
    0 \leq \r_i = \r_j' <\r_{j+1} \leq n- j - 1,
    $$
    so $0 \leq \r_j'<n-j-1 < n-j$. 
    Also, if $\r_j$ is not replaced, we have $0 \leq \r_j' = \r_j \leq n-j$.
    Thus $\r'$ is also the reversal of an inversion sequence.
\end{proof}

\begin{lemma}\label{lem2} Let $\r$ be an inversion sequence to which the algorithm is applied, and let $\r'$ be the resulting sequence. If $i$ is a transient occurrence of $p$, then $\r_{i+2}'$ is not equal to $\r_i'$.\end{lemma}

\begin{proof}
Since $i$ is a transient occurrence of $p$, we know by Lemma \ref{lem0} that this occurrence must have resulted from a replacement in position $i$, meaning that $i$ was the third position in a transient $p$ or original $p$ or $q$. (If $\r_i$ had not been changed, then $i$ is an original occurrence of $p$ or $q$.) Also by Lemma \ref{lem0}, this must have been the most recent replacement made when the algorithm reaches position $i+2$, so $\r_i$ is the entry stored in $\last$.

Since $i$ is a transient occurrence of $p$ or $q$, we know $\r_i' \neq \r_i$, and the entry in position $i+2$ will be replaced with the entry in $\last$, so $\r_{i+2}' = \r_i$. The result follows.
\end{proof}

\begin{theorem}\label{thm:sdexchange}
     Algorithm \ref{alg:sd} exchanges occurrences of $p$ and $q$ in a sequence $\r_1 \ldots \r_n$.
\end{theorem}
\begin{proof}
Let $\r'$ be the sequence that results after applying the algorithm to a sequence $\r$.
We will show that $ \Em(p,\r) = \Em(q,\r')$ and $\Em(q,\r) = \Em(p,\r')$. 

Suppose there is an original $p$ or $q$ in $\r$ at position $i$. If the first two entries in this occurrence were unchanged when the third digit is reached, then by construction, the algorithm will convert the pattern to $q$ or $p$, respectively.

By Lemma \ref{lem0}, the only other possibility is that the first digit was changed by the algorithm before the third digit was reached. By Lemma \ref{lem1}, even if the first digit was changed, it is still less than the second digit. So if $i$ is the position of an original $p$, and the entry in position $i$ was changed, we still have $\r_i' < \r_{i+1} \,(=\r_{i+1}')$ and $\r_{i+1} < \r_{i+3}\,(=\r_{i+3}')$. Then replacing $\r_{i+2}$ with $\r_{i+1}$ yields an occurrence of $q$ in position $i$, as the resulting sequence will have $\r_{i}' < \r_{i+1}' = \r_{i+2}' < \r_{i+3}'$. Similarly, if $i$ is the position of an original $q$, copying the entry in $\r_i$ to position $i+2$ will yield an occurrence of $p$ in position $i$.

So we have that $\Em(p,\r)\ci \Em(q,\r')$, and $\Em(q,\r)\ci  \Em(p,\r')$.

It remains to show that there will not be any occurrences of $p$ or $q$ in $\r'$ that were not in $\Em(p,\r)$ or $\Em(q,\r)$. By Lemma \ref{lem1cor}, we have that ``extra'' occurrences of $q$ cannot be generated by the algorithm.

If a new (transient) occurrence of $p$ is created at position $i$ in the sequence by the algorithm, this will be corrected by replacing the third entry in this occurrence with the entry stored in $\last$. 
This will change an occurrence of $p$ to a non-occurrence, as long as $\r_{i+2}'$ is not equal to $\r_i'$.
This is guaranteed by Lemma \ref{lem2}. 

So the resulting sequence $\r'$ will not have any occurrences of $p$ that were not in $\Em(q,\r)$, and will not have any occurrences of $q$ that were not in $\Em(p,\r)$. Thus the result holds. 
\end{proof}

\begin{lemma} \label{lem:replace}
The entry $\r_i' \neq \r_i$ $\iff$ $i-2$ is an original $p$, transient $p$, or original $q$.
\end{lemma}
\begin{proof}
    $(\Leftarrow)$: If $\r_i' \neq \r_i$ then it must be the case that $i$ was the third position of a pattern occurrence at some point in the sequence. (this is the only scenario where an entry could be changed).

     $(\Rightarrow)$: If $i-2$ is an original $p$, we have $\r_{i-1}' = \r_{i-1}$ (by Lemma \ref{lem0}), $\r_{i-1}>e_i$ and $\r_i' = \r_{i-1}$. So $\r_i' \neq \r_i$.

     If $i-2$ is an original $q$, we have $\r_{i-2}<\r_{i-1} = \r_i$ and $\r_i' = \r_{i-2}'$. If $\r_{i-2}' = \r_{i-2}$ then $\r_i' \neq \r_i$.

     If $\r_{i-2}' \neq \r_{i-2}$ then $\r_{i-2}'< \r_{i-1}$ still by Lemma \ref{lem1}.
     So $\r_i' =  \r_{i-2}'< \r_{i-1}= \r_i$; thus  $\r_i' \neq \r_i$.

     If $i-2$ is a transient $p$, we have $\r_{i-2} \neq \r_{i-2}' = \r_i$, and $\r_i' = \r_{i-2}$. So $\r_i' \neq \r_i$.
\end{proof}

\begin{theorem} \label{thm:sdinv}
     The mapping described by  Algorithm \ref{alg:sd} is an involution.
\end{theorem}

\begin{proof}
Let $\r$ be an inversion sequence, let $\r'$ be the resulting sequence after applying  Algorithm \ref{alg:sd} to $\r$, and let $\r''$ be the resulting sequence after applying the algorithm to $\r'$.
We will show by induction that $\r_i'' = \r_i$ for all $i$, so $\r'' = \r$.

Base case: The first entry will never be changed by the algorithm, because it is not the third entry of any pattern occurrence. So $\r_1 = \r_1' =\r_1''$.

Now assume $\r_{k}'' = \r_{k}$ for all $k<i$. We will show $\r_i'' = \r_i$.

\begin{enumerate}[(a)]
    \item First suppose that $\r_i' \neq \r_i$.
    
    By Lemma \ref{lem:replace}, Algorithm \ref{alg:sd} will only replace $\r_i$ if one of the following hold:
\begin{enumerate}[(i)]
    \item $\r_i$ is the third entry of an original $p$. Then $\r_i = \r_{i-2}$. Also by Theorem \ref{thm:sdexchange}, after applying the algorithm twice we know that $i-2$ will be an original $p$ in $\r''$, so $\r_i'' = \r_{i-2}''$. By the inductive hypothesis, $\r_{i-2}'' = \r_{i-2}$. So we have $\r_i = \r_{i-2} = \r_{i-2}'' = \r_i''$.
    \item $\r_i$ is the third entry of an original $q$. In this case, $\r_i = \r_{i-1}$. Also by Theorem \ref{thm:sdexchange}, after applying the algorithm twice we know $\r_i'' = \r_{i-1}''$. By the inductive hypothesis, $\r_{i-1}'' = \r_{i-1}$, so $\r_i = \r_{i-1} = \r_{i-1}'' = \r_i''$.
    \item $\r_i$ is the third entry of a transient $p$.  Since $i$ is third entry of transient $p$, we have $\r_{i-2}'=\r_i < \r_{i-1}<\r_{i+1}$. and $\r_{i-2} = \r_i'$. By the inductive hypothesis, $\r_{i-2}'' = \r_{i-2}$, so $\r_{i-2}'' = \r_i'$. Moreover, $\r_{i-1}' = \r_{i-1}$ and $\r_{i+1}' = \r_{i+1}$ by  Lemma \ref{lem0}.
So we have 
$$
\r_{i-2}'' = \r_i' = \r_{i-2} < \r_{i-1}=\r_{i-1}'< \r_{i+1}' = \r_{i+1}
$$
Thus $i-2$ is a transient $p$ in $\r'$. So $\r_i '' = \r_{i-2}'$; 
thus $\r_i = \r_{i-2}' = \r_i''$.
\end{enumerate}

So if $\r_i' \neq \r_i$, we have $\r_i'' = \r_i$.

\item Otherwise suppose $\r_i' = \r_i$. By Lemma \ref{lem:replace}, $i$ is not the third entry of a pattern occurrence that appears during the algorithm. We will show that this is also true in $\r'$, so $\r_i' = \r_i''$.

\begin{enumerate}[(i)]
    \item We know by Theorem \ref{thm:sdexchange} that $i-2$ is an original $p$ in $\r$ if and only if $i-2$ is an original $q$ in $\r'$ and $i-2$ is an original $q$ in $\r$ if and only if $i-2$ is an original $p$ in $\r'$ . It follows that $i-2$ is neither an original $p$ or $q$ in $\r'$, so in these cases, $\r_i' = \r_i''$.
    
    \item The only other possibility is that $i-2$ is a transient $p$ in $\r'$. Then $\r_{i-2}''=\r_i'<\r_{i-1}'<\r_{i+1}'$. 

We know 
\al{
\r_{i-2}'' &= \r_{i-2}, \tag{by the inductive hypothesis} \\
\r_{i-1}'' &= \r_{i-1}, \tag{by the inductive hypothesis}\\
\r_i' &= \r_i. \tag{by assumption}
}
so by substitution we have
$$
\r_{i-2}=\r_i<\r_{i-1}<\r_{i+1}'.
$$

\begin{enumerate}
    \item Now, if $\r_{i+1}=\r_{i+1}'$, we have 
$$
\r_{i-2}=\r_i<\r_{i-1}<\r_{i+1}
$$
which implies that there was an original $p$ in $\r$ at $i-2$, which is a contradiction.

\item Otherwise, $\r_{i+1}\neq \r_{i+1}'$, which by Lemma  \ref{lem:replace} implies that $i-1$ is an original $p$, original $q$, or transient $p$ in $\r$.
Then either $\r_{i+1}' = \r_{i}'$ (if $i-1$ is an original $p$), $\r_{i-1}'$ (if $i-1$ is an original $q$), or $\r_{i-1}$ (if $i-1$ is a transient $p$).

So $\r_{i+1}'$ is equal to either $\r_{i}' = \r_i$ or $\r_{i-1}' = \r_{i-1}$. However, we already had that $\r_i<\r_{i-1}<\r_{i+1}'$, so this is a contradiction.
\end{enumerate}

So $i-2$ is not a transient $p$ in $\r'$; thus $\r_i'' = \r_i'$.
\end{enumerate}
It follows that whenever $\r_i' = \r_i$, we have $\r_i'' = \r_i'$. So again we have $\r_i'' = \r_i$.
\end{enumerate}
In all possible cases, $\r_i'' = \r_i$, so the result holds by induction.
\end{proof}

\begin{cor}
Let $\r$ be an inversion sequence and let $\r'$ be the sequence resulting from applying Algorithm \ref{alg:sd} to $\r$. If there is a transient $p$ at $i$ in $\r$, then there is a transient $p$ at $i$ in $\r'$.
\end{cor} 

\begin{proof}
    This is show in case (iii) of part (a) of the proof of Theorem \ref{thm:sdinv}, relying on the hypothesis that $\r_k '' = \r_k$ for all $k<i$, which we now know is true for all $k$.
\end{proof}

Remark: This is not the same bijection as the one found recursively. For example, we see in Example 6 that the previous bijection maps 010223 to 011203. Algorithm \ref{alg:sd} would map 010223 to 011213. \\ 

\section{The remaining equivalences}
The remaining set of equivalences, $$
\underline{2010} \sim \underline{2110} \sim \underline{2120},
$$
can be shown similarly. Note that $2010$ is the reversal of the pattern $0102$, and $2110$ is the reversal of the pattern $0112$. So if a sequence $\r$ of length $n$ has an occurrence of $2010$ at position $i$, its reversal will have an occurrence of $0102$ at position $n-i-2$.

So to show the  equivalence $\underline{2010} \sim \underline{2110}$, we can modify the bijection by first reversing the input sequence, applying Algorithm \ref{alg:sd} to this, and then reversing the output sequence. By Corollary \ref{cor:reversal}, we know that this process will send inversion sequences to inversion sequences. So, this will give a bijection between $$\{ \r\in I_n : \Em(\underline{2010},\r) = S \text{ and } \Em(\underline{2110},\r) = T\}$$ and $$\{ \r\in I_n : \Em(\underline{2110},\r) = S \text{ and } \Em(\underline{2010},\r) = T \}$$ for all $S,T \ci [n]$.

Also note that converting between $\underline{2120}$ and $\underline{2110}$ requires the third entry to be replaced by the value in either the first or second entry, similarly to the case for switching between $\underline{0102}$ and $\underline{0112}$. Indeed, reassigning $p = \underline{2120}$ and $q = \underline{2110}$, Algorithm \ref{alg:sd} as described previously yields a bijection between $\{ \r\in I_n : \Em(\underline{2120},\r) = S \text{ and } \Em(\underline{2110},\r) = T\}$ and $\{ \r\in I_n : \Em(\underline{2110},\r) = S \text{ and } \Em(\underline{2120},\r) = T \}$ for all $(S,T) \ci [n] $. In order to verify this, we will need to make slight modifications in the proofs of the previous lemmas. First, Lemma \ref{lem0} still holds since \underline{2120} and \underline{2110} can only overlap themselves or each other in one or two entries.

For $p = \underline{2120}$ and $q = \underline{2110}$, we define original and transient pattern occurrences analogously as before:

If $\r$ is the initial sequence and $\r'$ is the sequence resulting from applying the algorithm to $\r$, we say there is
\begin{itemize}
    \item an \textit{original} occurrence of $p$ at $i$ if $\r_i = \r_{i+2} >\r_{i+1}>\r_{i+3}$,
    \item an \textit{original} occurrence of $q$ at $i$ if $\r_i >\r_{i+1}= \r_{i+2} >\r_{i+3}$,
    \item  a \textit{transient} occurrence of $p$ at $i$ if $i$ is not an original occurrence of $p$, but $\r_i' = \r_{i+2}>\r_{i+1}>\r_{i+3}$, and
    \item a \textit{transient} occurrence of $q$ at $i$  if $i$ is not an original occurrence of $q$, but $\r_i' >\r_{i+1}= \r_{i+2} >\r_{i+3}$.
\end{itemize} 

Note: These are the same definitions we gave before, with all inequalities reversed.
The results shown for the case when $p = \underline{0102}$ and $q= \underline{0112}$ have analogs for $p =\underline{2120}$ and $q=\underline{2110}$, obtained by reversing the inequalities. So, Algorithm \ref{alg:sd} still works to show that $\underline{2120} \recip \underline{2110}$.

Lastly, we note that collecting all of these results, we have reciprocal relations among almost all pairs given in Theorem \ref{asupdated}, with the exceptions of 
 $\underline{2010} \sim \underline{2120}$, $\underline{1000} \sim \underline{1110}$, and $\underline{2100} \sim \underline{2210}$.
These cannot be reciprocal due to asymmetry in how the patterns can overlap:

For $\underline{2120}$ and $\underline{2010}$, the sequence $00212010$ has $\underline{2120}$ at position $3$ and $\underline{2010}$ at position $5$. However, there is no inversion sequence $\r$ with
$\underline{2010}$ at position $3$ and $\underline{2120}$ at position $5$.
Having $\underline{2120}$ at position $5$ would require $\r_5 > \r_6 > \r_8$, and also having $\underline{2010}$ at position $3$ would require $\r_3 > \r_5$, so combined, we have $\r_3 > \r_5 > \r_6 > \r_8$. Since the entries in inversion sequences are all nonnegative integers, it follows that $\r_3 \geq 3$. However, if $\r$ is an inversion sequence, we must have $\r_3 < 3$, so this is not possible. 

The sequence ${21110}$ has $\underline{1000}$ at position 1 and $\underline{1110}$ at position 2.
However, we cannot have a sequence with $\underline{1000}$ at position 2 and $\underline{1110}$ at position 1, since the $\underline{1110}$ at position 1 would require that the second and third entries of the sequence are equal, but having $\underline{1000}$ at position 2 would require that the second and third entries of the sequence are not equal.

Similarly, $22100$ has $\underline{2100}$ at position 2 and $ \underline{2210}$ at position 1. However, we cannot create a sequence with $\underline{2100}$ at position 1 and $ \underline{2210}$ at position 2 --- the first condition would require the second and third entries to be different, but the second condition requires them to be equal.

\newpage

\appendix 
\appendixtitleon
\begin{appendices}
\section{Extended example}\label{extendedexample}
 \begin{example}
$\r = 011203140516$ \\ 
We currently have $p$ at $\{ \}$; we want $q$ at $\{ \}$. \\ 
$\ph_{ = \{ \} } (011203140516)$: \\ 
\hphantom{...}$\mid$\,\,$\ph_{ \geq \{ \} } (011203140516) = 011203140516$ \\ 
\hphantom{...}$\mid$\,\,$q$ occurs at $\{ 1\}$, so apply $\ph\inv_{ = \{ 1\} }$ \\ 
\hphantom{...}$\mid$\,\,$\ph\inv_{ = \{ 1\} } (011203140516)$: \\ 
\hphantom{...}$\mid$\,\,\hphantom{...}$\mid$\,\,$\ph\inv_{ \geq \{ 1\} } (011203140516) = 010203140516$ \\ 
\hphantom{...}$\mid$\,\,\hphantom{...}$\mid$\,\,$p$ occurs at $\{ 1, 3\}$, so apply $\ph_{ = \{ 1, 3\} }$ \\ 
\hphantom{...}$\mid$\,\,\hphantom{...}$\mid$\,\,$\ph_{ = \{ 1, 3\} } (010203140516)$: \\ 
\hphantom{...}$\mid$\,\,\hphantom{...}$\mid$\,\,\hphantom{...}$\mid$\,\,$\ph_{ \geq \{ 1, 3\} } (010203140516) = 011223140516$ \\ 
\hphantom{...}$\mid$\,\,\hphantom{...}$\mid$\,\,\hphantom{...}$\mid$\,\,$q$ occurs at $\{ 1, 3\}$, so we are done. \\ 
\hphantom{...}$\mid$\,\,\hphantom{...}$\mid$\,\,$011223140516$ \\ 
\hphantom{...}$\mid$\,\,\hphantom{...}$\mid$\,\,$\ph\inv_{ \geq \{ 1\} } (011223140516) = 010223140516$ \\ 
\hphantom{...}$\mid$\,\,\hphantom{...}$\mid$\,\,$p$ occurs at $\{ 1\}$, so we are done. \\ 
\hphantom{...}$\mid$\,\,$010223140516$ \\ 
\hphantom{...}$\mid$\,\,$\ph_{ \geq \{ \} } (010223140516) = 010223140516$ \\ 
\hphantom{...}$\mid$\,\,$q$ occurs at $\{ 3\}$, so apply $\ph\inv_{ = \{ 3\} }$ \\ 
\hphantom{...}$\mid$\,\,$\ph\inv_{ = \{ 3\} } (010223140516)$: \\ 
\hphantom{...}$\mid$\,\,\hphantom{...}$\mid$\,\,$\ph\inv_{ \geq \{ 3\} } (010223140516) = 010203140516$ \\ 
\hphantom{...}$\mid$\,\,\hphantom{...}$\mid$\,\,$p$ occurs at $\{ 1, 3\}$, so apply $\ph_{ = \{ 1, 3\} }$ \\ 
\hphantom{...}$\mid$\,\,\hphantom{...}$\mid$\,\,$\ph_{ = \{ 1, 3\} } (010203140516)$: \\ 
\hphantom{...}$\mid$\,\,\hphantom{...}$\mid$\,\,\hphantom{...}$\mid$\,\,$\ph_{ \geq \{ 1, 3\} } (010203140516) = 011223140516$ \\ 
\hphantom{...}$\mid$\,\,\hphantom{...}$\mid$\,\,\hphantom{...}$\mid$\,\,$q$ occurs at $\{ 1, 3\}$, so we are done. \\ 
\hphantom{...}$\mid$\,\,\hphantom{...}$\mid$\,\,$011223140516$ \\ 
\hphantom{...}$\mid$\,\,\hphantom{...}$\mid$\,\,$\ph\inv_{ \geq \{ 3\} } (011223140516) = 011213140516$ \\ 
\hphantom{...}$\mid$\,\,\hphantom{...}$\mid$\,\,$p$ occurs at $\{ 3, 5\}$, so apply $\ph_{ = \{ 3, 5\} }$ \\ 
\hphantom{...}$\mid$\,\,\hphantom{...}$\mid$\,\,$\ph_{ = \{ 3, 5\} } (011213140516)$: \\ 
\hphantom{...}$\mid$\,\,\hphantom{...}$\mid$\,\,\hphantom{...}$\mid$\,\,$\ph_{ \geq \{ 3, 5\} } (011213140516) = 011223340516$ \\ 
\hphantom{...}$\mid$\,\,\hphantom{...}$\mid$\,\,\hphantom{...}$\mid$\,\,$q$ occurs at $\{ 1, 3, 5\}$, so apply $\ph\inv_{ = \{ 1, 3, 5\} }$ \\ 
\hphantom{...}$\mid$\,\,\hphantom{...}$\mid$\,\,\hphantom{...}$\mid$\,\,$\ph\inv_{ = \{ 1, 3, 5\} } (011223340516)$: \\ 
\hphantom{...}$\mid$\,\,\hphantom{...}$\mid$\,\,\hphantom{...}$\mid$\,\,\hphantom{...}$\mid$\,\,$\ph\inv_{ \geq \{ 1, 3, 5\} } (011223340516) = 010203040516$ \\ 
\hphantom{...}$\mid$\,\,\hphantom{...}$\mid$\,\,\hphantom{...}$\mid$\,\,\hphantom{...}$\mid$\,\,$p$ occurs at $\{ 1, 3, 5, 7\}$, so apply $\ph_{ = \{ 1, 3, 5, 7\} }$ \\ 
\hphantom{...}$\mid$\,\,\hphantom{...}$\mid$\,\,\hphantom{...}$\mid$\,\,\hphantom{...}$\mid$\,\,$\ph_{ = \{ 1, 3, 5, 7\} } (010203040516)$: \\ 
\hphantom{...}$\mid$\,\,\hphantom{...}$\mid$\,\,\hphantom{...}$\mid$\,\,\hphantom{...}$\mid$\,\,\hphantom{...}$\mid$\,\,$\ph_{ \geq \{ 1, 3, 5, 7\} } (010203040516) = 011223344516$ \\ 
\hphantom{...}$\mid$\,\,\hphantom{...}$\mid$\,\,\hphantom{...}$\mid$\,\,\hphantom{...}$\mid$\,\,\hphantom{...}$\mid$\,\,$q$ occurs at $\{ 1, 3, 5, 7\}$, so we are done. \\ 
\hphantom{...}$\mid$\,\,\hphantom{...}$\mid$\,\,\hphantom{...}$\mid$\,\,\hphantom{...}$\mid$\,\,$011223344516$ \\ 
\hphantom{...}$\mid$\,\,\hphantom{...}$\mid$\,\,\hphantom{...}$\mid$\,\,\hphantom{...}$\mid$\,\,$\ph\inv_{ \geq \{ 1, 3, 5\} } (011223344516) = 010203044516$ \\ 
\hphantom{...}$\mid$\,\,\hphantom{...}$\mid$\,\,\hphantom{...}$\mid$\,\,\hphantom{...}$\mid$\,\,$p$ occurs at $\{ 1, 3, 5\}$, so we are done. \\ 
\hphantom{...}$\mid$\,\,\hphantom{...}$\mid$\,\,\hphantom{...}$\mid$\,\,$010203044516$ \\ 
\hphantom{...}$\mid$\,\,\hphantom{...}$\mid$\,\,\hphantom{...}$\mid$\,\,$\ph_{ \geq \{ 3, 5\} } (010203044516) = 010223344516$ \\ 
\hphantom{...}$\mid$\,\,\hphantom{...}$\mid$\,\,\hphantom{...}$\mid$\,\,$q$ occurs at $\{ 3, 5, 7\}$, so apply $\ph\inv_{ = \{ 3, 5, 7\} }$ \\ 
\hphantom{...}$\mid$\,\,\hphantom{...}$\mid$\,\,\hphantom{...}$\mid$\,\,$\ph\inv_{ = \{ 3, 5, 7\} } (010223344516)$: \\ 
\hphantom{...}$\mid$\,\,\hphantom{...}$\mid$\,\,\hphantom{...}$\mid$\,\,\hphantom{...}$\mid$\,\,$\ph\inv_{ \geq \{ 3, 5, 7\} } (010223344516) = 010203040516$ \\ 
\hphantom{...}$\mid$\,\,\hphantom{...}$\mid$\,\,\hphantom{...}$\mid$\,\,\hphantom{...}$\mid$\,\,$p$ occurs at $\{ 1, 3, 5, 7\}$, so apply $\ph_{ = \{ 1, 3, 5, 7\} }$ \\ 
\hphantom{...}$\mid$\,\,\hphantom{...}$\mid$\,\,\hphantom{...}$\mid$\,\,\hphantom{...}$\mid$\,\,$\ph_{ = \{ 1, 3, 5, 7\} } (010203040516)$: \\ 
\hphantom{...}$\mid$\,\,\hphantom{...}$\mid$\,\,\hphantom{...}$\mid$\,\,\hphantom{...}$\mid$\,\,\hphantom{...}$\mid$\,\,$\ph_{ \geq \{ 1, 3, 5, 7\} } (010203040516) = 011223344516$ \\ 
\hphantom{...}$\mid$\,\,\hphantom{...}$\mid$\,\,\hphantom{...}$\mid$\,\,\hphantom{...}$\mid$\,\,\hphantom{...}$\mid$\,\,$q$ occurs at $\{ 1, 3, 5, 7\}$, so we are done. \\ 
\hphantom{...}$\mid$\,\,\hphantom{...}$\mid$\,\,\hphantom{...}$\mid$\,\,\hphantom{...}$\mid$\,\,$011223344516$ \\ 
\hphantom{...}$\mid$\,\,\hphantom{...}$\mid$\,\,\hphantom{...}$\mid$\,\,\hphantom{...}$\mid$\,\,$\ph\inv_{ \geq \{ 3, 5, 7\} } (011223344516) = 011213141516$ \\ 
\hphantom{...}$\mid$\,\,\hphantom{...}$\mid$\,\,\hphantom{...}$\mid$\,\,\hphantom{...}$\mid$\,\,$p$ occurs at $\{ 3, 5, 7, 9\}$, so apply $\ph_{ = \{ 3, 5, 7, 9\} }$ \\ 
\hphantom{...}$\mid$\,\,\hphantom{...}$\mid$\,\,\hphantom{...}$\mid$\,\,\hphantom{...}$\mid$\,\,$\ph_{ = \{ 3, 5, 7, 9\} } (011213141516)$: \\ 
\hphantom{...}$\mid$\,\,\hphantom{...}$\mid$\,\,\hphantom{...}$\mid$\,\,\hphantom{...}$\mid$\,\,\hphantom{...}$\mid$\,\,$\ph_{ \geq \{ 3, 5, 7, 9\} } (011213141516) = 011223344556$ \\ 
\hphantom{...}$\mid$\,\,\hphantom{...}$\mid$\,\,\hphantom{...}$\mid$\,\,\hphantom{...}$\mid$\,\,\hphantom{...}$\mid$\,\,$q$ occurs at $\{ 1, 3, 5, 7, 9\}$, so apply $\ph\inv_{ = \{ 1, 3, 5, 7, 9\} }$ \\ 
\hphantom{...}$\mid$\,\,\hphantom{...}$\mid$\,\,\hphantom{...}$\mid$\,\,\hphantom{...}$\mid$\,\,\hphantom{...}$\mid$\,\,$\ph\inv_{ = \{ 1, 3, 5, 7, 9\} } (011223344556)$: \\ 
\hphantom{...}$\mid$\,\,\hphantom{...}$\mid$\,\,\hphantom{...}$\mid$\,\,\hphantom{...}$\mid$\,\,\hphantom{...}$\mid$\,\,\hphantom{...}$\mid$\,\,$\ph\inv_{ \geq \{ 1, 3, 5, 7, 9\} } (011223344556) = 010203040506$ \\ 
\hphantom{...}$\mid$\,\,\hphantom{...}$\mid$\,\,\hphantom{...}$\mid$\,\,\hphantom{...}$\mid$\,\,\hphantom{...}$\mid$\,\,\hphantom{...}$\mid$\,\,$p$ occurs at $\{ 1, 3, 5, 7, 9\}$, so we are done. \\ 
\hphantom{...}$\mid$\,\,\hphantom{...}$\mid$\,\,\hphantom{...}$\mid$\,\,\hphantom{...}$\mid$\,\,\hphantom{...}$\mid$\,\,$010203040506$ \\ 
\hphantom{...}$\mid$\,\,\hphantom{...}$\mid$\,\,\hphantom{...}$\mid$\,\,\hphantom{...}$\mid$\,\,\hphantom{...}$\mid$\,\,$\ph_{ \geq \{ 3, 5, 7, 9\} } (010203040506) = 010223344556$ \\ 
\hphantom{...}$\mid$\,\,\hphantom{...}$\mid$\,\,\hphantom{...}$\mid$\,\,\hphantom{...}$\mid$\,\,\hphantom{...}$\mid$\,\,$q$ occurs at $\{ 3, 5, 7, 9\}$, so we are done. \\ 
\hphantom{...}$\mid$\,\,\hphantom{...}$\mid$\,\,\hphantom{...}$\mid$\,\,\hphantom{...}$\mid$\,\,$010223344556$ \\ 
\hphantom{...}$\mid$\,\,\hphantom{...}$\mid$\,\,\hphantom{...}$\mid$\,\,\hphantom{...}$\mid$\,\,$\ph\inv_{ \geq \{ 3, 5, 7\} } (010223344556) = 010203040556$ \\ 
\hphantom{...}$\mid$\,\,\hphantom{...}$\mid$\,\,\hphantom{...}$\mid$\,\,\hphantom{...}$\mid$\,\,$p$ occurs at $\{ 1, 3, 5, 7\}$, so apply $\ph_{ = \{ 1, 3, 5, 7\} }$ \\ 
\hphantom{...}$\mid$\,\,\hphantom{...}$\mid$\,\,\hphantom{...}$\mid$\,\,\hphantom{...}$\mid$\,\,$\ph_{ = \{ 1, 3, 5, 7\} } (010203040556)$: \\ 
\hphantom{...}$\mid$\,\,\hphantom{...}$\mid$\,\,\hphantom{...}$\mid$\,\,\hphantom{...}$\mid$\,\,\hphantom{...}$\mid$\,\,$\ph_{ \geq \{ 1, 3, 5, 7\} } (010203040556) = 011223344556$ \\ 
\hphantom{...}$\mid$\,\,\hphantom{...}$\mid$\,\,\hphantom{...}$\mid$\,\,\hphantom{...}$\mid$\,\,\hphantom{...}$\mid$\,\,$q$ occurs at $\{ 1, 3, 5, 7, 9\}$, so apply $\ph\inv_{ = \{ 1, 3, 5, 7, 9\} }$ \\ 
\hphantom{...}$\mid$\,\,\hphantom{...}$\mid$\,\,\hphantom{...}$\mid$\,\,\hphantom{...}$\mid$\,\,\hphantom{...}$\mid$\,\,$\ph\inv_{ = \{ 1, 3, 5, 7, 9\} } (011223344556)$: \\ 
\hphantom{...}$\mid$\,\,\hphantom{...}$\mid$\,\,\hphantom{...}$\mid$\,\,\hphantom{...}$\mid$\,\,\hphantom{...}$\mid$\,\,\hphantom{...}$\mid$\,\,$\ph\inv_{ \geq \{ 1, 3, 5, 7, 9\} } (011223344556) = 010203040506$ \\ 
\hphantom{...}$\mid$\,\,\hphantom{...}$\mid$\,\,\hphantom{...}$\mid$\,\,\hphantom{...}$\mid$\,\,\hphantom{...}$\mid$\,\,\hphantom{...}$\mid$\,\,$p$ occurs at $\{ 1, 3, 5, 7, 9\}$, so we are done. \\ 
\hphantom{...}$\mid$\,\,\hphantom{...}$\mid$\,\,\hphantom{...}$\mid$\,\,\hphantom{...}$\mid$\,\,\hphantom{...}$\mid$\,\,$010203040506$ \\ 
\hphantom{...}$\mid$\,\,\hphantom{...}$\mid$\,\,\hphantom{...}$\mid$\,\,\hphantom{...}$\mid$\,\,\hphantom{...}$\mid$\,\,$\ph_{ \geq \{ 1, 3, 5, 7\} } (010203040506) = 011223344506$ \\ 
\hphantom{...}$\mid$\,\,\hphantom{...}$\mid$\,\,\hphantom{...}$\mid$\,\,\hphantom{...}$\mid$\,\,\hphantom{...}$\mid$\,\,$q$ occurs at $\{ 1, 3, 5, 7\}$, so we are done. \\ 
\hphantom{...}$\mid$\,\,\hphantom{...}$\mid$\,\,\hphantom{...}$\mid$\,\,\hphantom{...}$\mid$\,\,$011223344506$ \\ 
\hphantom{...}$\mid$\,\,\hphantom{...}$\mid$\,\,\hphantom{...}$\mid$\,\,\hphantom{...}$\mid$\,\,$\ph\inv_{ \geq \{ 3, 5, 7\} } (011223344506) = 011213141506$ \\ 
\hphantom{...}$\mid$\,\,\hphantom{...}$\mid$\,\,\hphantom{...}$\mid$\,\,\hphantom{...}$\mid$\,\,$p$ occurs at $\{ 3, 5, 7\}$, so we are done. \\ 
\hphantom{...}$\mid$\,\,\hphantom{...}$\mid$\,\,\hphantom{...}$\mid$\,\,$011213141506$ \\ 
\hphantom{...}$\mid$\,\,\hphantom{...}$\mid$\,\,\hphantom{...}$\mid$\,\,$\ph_{ \geq \{ 3, 5\} } (011213141506) = 011223341506$ \\ 
\hphantom{...}$\mid$\,\,\hphantom{...}$\mid$\,\,\hphantom{...}$\mid$\,\,$q$ occurs at $\{ 1, 3, 5\}$, so apply $\ph\inv_{ = \{ 1, 3, 5\} }$ \\ 
\hphantom{...}$\mid$\,\,\hphantom{...}$\mid$\,\,\hphantom{...}$\mid$\,\,$\ph\inv_{ = \{ 1, 3, 5\} } (011223341506)$: \\ 
\hphantom{...}$\mid$\,\,\hphantom{...}$\mid$\,\,\hphantom{...}$\mid$\,\,\hphantom{...}$\mid$\,\,$\ph\inv_{ \geq \{ 1, 3, 5\} } (011223341506) = 010203041506$ \\ 
\hphantom{...}$\mid$\,\,\hphantom{...}$\mid$\,\,\hphantom{...}$\mid$\,\,\hphantom{...}$\mid$\,\,$p$ occurs at $\{ 1, 3, 5\}$, so we are done. \\ 
\hphantom{...}$\mid$\,\,\hphantom{...}$\mid$\,\,\hphantom{...}$\mid$\,\,$010203041506$ \\ 
\hphantom{...}$\mid$\,\,\hphantom{...}$\mid$\,\,\hphantom{...}$\mid$\,\,$\ph_{ \geq \{ 3, 5\} } (010203041506) = 010223341506$ \\ 
\hphantom{...}$\mid$\,\,\hphantom{...}$\mid$\,\,\hphantom{...}$\mid$\,\,$q$ occurs at $\{ 3, 5\}$, so we are done. \\ 
\hphantom{...}$\mid$\,\,\hphantom{...}$\mid$\,\,$010223341506$ \\ 
\hphantom{...}$\mid$\,\,\hphantom{...}$\mid$\,\,$\ph\inv_{ \geq \{ 3\} } (010223341506) = 010203341506$ \\ 
\hphantom{...}$\mid$\,\,\hphantom{...}$\mid$\,\,$p$ occurs at $\{ 1, 3\}$, so apply $\ph_{ = \{ 1, 3\} }$ \\ 
\hphantom{...}$\mid$\,\,\hphantom{...}$\mid$\,\,$\ph_{ = \{ 1, 3\} } (010203341506)$: \\ 
\hphantom{...}$\mid$\,\,\hphantom{...}$\mid$\,\,\hphantom{...}$\mid$\,\,$\ph_{ \geq \{ 1, 3\} } (010203341506) = 011223341506$ \\ 
\hphantom{...}$\mid$\,\,\hphantom{...}$\mid$\,\,\hphantom{...}$\mid$\,\,$q$ occurs at $\{ 1, 3, 5\}$, so apply $\ph\inv_{ = \{ 1, 3, 5\} }$ \\ 
\hphantom{...}$\mid$\,\,\hphantom{...}$\mid$\,\,\hphantom{...}$\mid$\,\,$\ph\inv_{ = \{ 1, 3, 5\} } (011223341506)$: \\ 
\hphantom{...}$\mid$\,\,\hphantom{...}$\mid$\,\,\hphantom{...}$\mid$\,\,\hphantom{...}$\mid$\,\,$\ph\inv_{ \geq \{ 1, 3, 5\} } (011223341506) = 010203041506$ \\ 
\hphantom{...}$\mid$\,\,\hphantom{...}$\mid$\,\,\hphantom{...}$\mid$\,\,\hphantom{...}$\mid$\,\,$p$ occurs at $\{ 1, 3, 5\}$, so we are done. \\ 
\hphantom{...}$\mid$\,\,\hphantom{...}$\mid$\,\,\hphantom{...}$\mid$\,\,$010203041506$ \\ 
\hphantom{...}$\mid$\,\,\hphantom{...}$\mid$\,\,\hphantom{...}$\mid$\,\,$\ph_{ \geq \{ 1, 3\} } (010203041506) = 011223041506$ \\ 
\hphantom{...}$\mid$\,\,\hphantom{...}$\mid$\,\,\hphantom{...}$\mid$\,\,$q$ occurs at $\{ 1, 3\}$, so we are done. \\ 
\hphantom{...}$\mid$\,\,\hphantom{...}$\mid$\,\,$011223041506$ \\ 
\hphantom{...}$\mid$\,\,\hphantom{...}$\mid$\,\,$\ph\inv_{ \geq \{ 3\} } (011223041506) = 011213041506$ \\ 
\hphantom{...}$\mid$\,\,\hphantom{...}$\mid$\,\,$p$ occurs at $\{ 3\}$, so we are done. \\ 
\hphantom{...}$\mid$\,\,$011213041506$ \\ 
\hphantom{...}$\mid$\,\,$\ph_{ \geq \{ \} } (011213041506) = 011213041506$ \\ 
\hphantom{...}$\mid$\,\,$q$ occurs at $\{ 1\}$, so apply $\ph\inv_{ = \{ 1\} }$ \\ 
\hphantom{...}$\mid$\,\,$\ph\inv_{ = \{ 1\} } (011213041506)$: \\ 
\hphantom{...}$\mid$\,\,\hphantom{...}$\mid$\,\,$\ph\inv_{ \geq \{ 1\} } (011213041506) = 010213041506$ \\ 
\hphantom{...}$\mid$\,\,\hphantom{...}$\mid$\,\,$p$ occurs at $\{ 1\}$, so we are done. \\ 
\hphantom{...}$\mid$\,\,$010213041506$ \\ 
\hphantom{...}$\mid$\,\,$\ph_{ \geq \{ \} } (010213041506) = 010213041506$ \\ 
\hphantom{...}$\mid$\,\,$q$ occurs at $\{ \}$, so we are done. \\ 
$010213041506$ \\ 
We finally obtain $011203140516 \mapsto 010213041506$.  \\ 
 \end{example}

\end{appendices}

\newpage
\printbibliography

\end{document}